\documentclass[a4paper,11pt]{article}

\usepackage[T1]{fontenc}
\usepackage{lmodern}

\usepackage{amsmath, amssymb, amsthm}
\usepackage[margin=2.5cm]{geometry}
\usepackage[svgnames,x11names,table]{xcolor}
\usepackage{subcaption}
\usepackage{graphicx}
\usepackage{tikz}
\usetikzlibrary{positioning}
\usepackage{mathtools}

\usepackage{bbm}
\usepackage{framed}  
\usepackage[ruled,vlined]{algorithm2e}
\usepackage{booktabs}
\usepackage{cite}

\usepackage{etoolbox}
\makeatletter
\patchcmd{\@maketitle}{\LARGE \@title}{\LARGE\bfseries\@title}{}{}

\renewcommand{\@seccntformat}[1]{\csname the#1\endcsname.\quad}
\makeatother

\definecolor{darkblue}{rgb}{0,0,.5}

\usepackage{hyperref} 
\hypersetup{
        colorlinks=true,   
        linkcolor=darkblue,    
        citecolor=darkblue,    
        urlcolor=darkblue      
      }

\usepackage{xfrac}
\usepackage{microtype}
\renewcommand{\leq}{\leqslant}
\renewcommand{\geq}{\geqslant}
\renewcommand{\emptyset}{\varnothing}

\newtheorem{theorem}{Theorem}[section]
\newtheorem{lemma}[theorem]{Lemma}
\newtheorem{corollary}[theorem]{Corollary}

\newtheorem{assumption}[theorem]{Assumption}
\theoremstyle{definition}
\newtheorem{definition}[theorem]{Definition}
\theoremstyle{definition}

\theoremstyle{definition}
\newtheorem{remark}[theorem]{Remark}

\usepackage[shortlabels]{enumitem}

\allowdisplaybreaks

\newcounter{step}[algocf]
\newcommand\step[1]{%
    \refstepcounter{step}    
    \vskip 0.25\baselineskip
    \ifx\hfuzz#1\hfuzz
        \noindent\(\triangleright\)~\textbf{Step~\arabic{step}.}%
    \else
        \noindent\(\triangleright\)~\textbf{Step~\arabic{step}}~(\texttt{#1})\textbf{.}%
    \fi
}

\newcommand{\R}{\mathbbm R}
\newcommand{\N}{\mathbbm N}
\newcommand{\Z}{\mathbbm Z}

\DeclareMathOperator{\tr}{tr}
\DeclareMathOperator{\prox}{prox}
\DeclareMathOperator{\proj}{proj}
\DeclareMathOperator{\dom}{dom}

\DeclareMathOperator{\range}{range}

\DeclareMathOperator*{\argmin}{arg\,min}

\newcommand{\bx}{\mathbf{x}}
\newcommand{\by}{\mathbf{y}}

\newcommand{\bu}{\mathbf{u}}
\newcommand{\bw}{\mathbf{w}}

\newcommand{\deltaB}{\delta_K}
\newcommand{\deltaL}{\delta_L}

\newcommand{\vertiii}[1]{{\left\vert\kern-0.25ex\left\vert\kern-0.25ex\left\vert #1 
    \right\vert\kern-0.25ex\right\vert\kern-0.25ex\right\vert}}

\renewcommand{\mu}{\rho}

\begin{document}

\title{A distributed proximal splitting method with linesearch for {locally Lipschitz gradients}}

\author{Felipe Atenas\thanks{School of Mathematics \& Statistics,
                             The University of Melbourne, Australia.
                             E-mail:~\href{href:felipe.atenas@unimelb.edu.au}{felipe.atenas@unimelb.edu.au}}
        \and
        Minh N.\ Dao\thanks{School of Science,
                            RMIT University, Australia
                            E-mail:~\href{href:minh.dao@rmit.edu.au}{minh.dao@rmit.edu.au}}
        \and
        Matthew K.\ Tam\thanks{School of Mathematics \& Statistics,
                               The University of Melbourne, Australia.
                               E-mail:~\href{href:matthew.tam@unimelb.edu.au}{matthew.tam@unimelb.edu.au}}    }

\date{\today}

\maketitle

\begin{abstract} 
In this paper, we propose a distributed first-order algorithm with backtracking linesearch for solving multi-agent minimisation problems, where each agent handles a local objective involving nonsmooth and smooth components. Unlike existing methods that require global Lipschitz continuity and predefined stepsizes, our algorithm adjusts stepsizes using distributed linesearch procedures, making it suitable for problems where global constants are unavailable or difficult to compute. The proposed algorithm is designed within an abstract linesearch framework for a primal-dual proximal-gradient method to solve min-max convex-concave problems, enabling the consensus constraint to be decoupled from the optimisation task. Our theoretical analysis allows for gradients of functions to be locally Lipschitz continuous, relaxing the prevalent assumption of globally Lipschitz continuous gradients. 
\end{abstract}

\paragraph{Keywords:}
Backtracking linesearch,
distributed optimisation,
locally Lipschitz data,
min-max problems,
primal-dual algorithm,
proximal splitting algorithm.

\paragraph{Mathematics Subject Classification (MSC 2020):}
90C25,	
68W15,  
49M27,	
65K05.	

\section{Introduction}

In this work, we consider a network of $n$ agents, connected by a communication graph, working collaboratively to solve  \begin{equation} \label{primal-problem}
    \min_{x \in \R^d} \sum_{i = 1}^n (f_i(x) + h_i(x)),
\end{equation} 
where each agent $i \in \{1, \dots, n\}$ possesses its own components of the objective function $f_i : \R^{d} \to \R \cup \{+\infty\}$ and $h_i : \R^{d} \to \R$. Here, $f_i$ is a proper lower semicontinuous convex function that may represent constraints or nonsmooth components of the problem, while $h_i$ is a differentiable convex function with locally Lipschitz continuous gradient capturing smooth components in the objective. Each agent has the ability to perform local computations using their respective $f_i$ and $h_i$, and the agents communicate via their immediate neighbours in the network to exchange information.

Distributed optimisation problems fitting the form of \eqref{primal-problem} arise in numerous applications including signal processing, statistics and machine learning \cite{mateos2010distributed,ravazzi2015distributed,nedic2017fast,omidshafiei2017deep}. Several methods have been proposed to address problems of this type; however, they typically assume that smooth components in the objective function have globally Lipschitz continuous gradients \cite{shi2015extra,shi2015proximal, li2019decentralized}, leading to convergence guarantees involving stepsize bounds  that depend on the inverse of the (global) Lipschitz modulus \cite{combettes2011proximal, shi2015proximal, li2019decentralized}. In the case such constants are unavailable, for instance, when the problem formulation does not have globally Lipschitz gradients, or when the Lipschitz constant exists but cannot be computed accurately, defining appropriate stepsizes poses a challenge. Consequently, alternative strategies for determining the stepsizes, such as via a backtracking linesearch~\cite{beck2010gradient, goldstein2013adaptive, goldstein2015adaptive, bello2016convergence}, are required.

In order to solve \eqref{primal-problem}, we present linesearch procedures for a distributed first-order algorithm. In the proposed method, each component of the objective function is handled separately, and each iteration utilises a backtracking linesearch involving communication, only allowing each agent to communicate with immediate neighbours and update local variables. For each individual agent, vector operations are performed in a decentralised manner, and the coordination operations to determine the stepsizes use distributed communication with only scalar values. Our approach builds on decentralised proximal-gradient methods, casting problem \eqref{primal-problem} into the more general framework of min-max problems and addressing the consensus constraint in the dual space. Primal-dual methods to solve min-max problems have been extensively studied in the literature, see  \cite{chambolle2011first, he2012convergence, komodakis2015playing, li2021new, malitsky2023first} and references therein.

Let us recall that a \emph{proximal-gradient exact first-order algorithm} (PG-EXTRA) is proposed in \cite{shi2015proximal} to solve \eqref{primal-problem}. This algorithm performs rounds of communications for the iterates in each iteration, using a common constant stepsize among agents. By denoting $N(i)$ the set of agents directly connected to agent $i$ in the network and $W \in \R^{n \times n}$ a matrix that represents the topology of the network and models communication among agents, called a \emph{mixing matrix} (see Definition~\ref{def:mixing-matrix}), the main iteration of PG-EXTRA can be formulated as follows. Given a fixed stepsize $\sigma>0$, for each $k \geq 1$ and $i \in \{1, \dots ,n\}$, compute 

\begin{equation} \label{agent-PGEXTRA}
    \left\{\begin{aligned}
         w^k_i& = w^{k-1}_i + \displaystyle\sum_{j \in N(i)}W_{ij}x^k_j - \dfrac{1}{2}\left(\displaystyle\sum_{j \in N(i)} W_{ij} x_j^{k-1} + x_i^{k-1}  \right) - \sigma\left(\nabla h_i(x^k_i) - \nabla h_i(x^{k-1}_i)\right)\\
         x^{k+1}_i & = \prox_{\sigma  f_i}\big(w^k_i\big).
    \end{aligned}\right.
\end{equation}
This algorithm converges to a solution to \eqref{primal-problem} as long as each $h_i$ has globally Lipschitz continuous gradients with constant $L_i$, and the stepsize $\sigma$ satisfies a bound that depends on the minimum eigenvalue of $W$ and the Lipschitz constants: 
\begin{equation*} 
    0 < \sigma < \dfrac{1 + \lambda_{\min}(W)}{L_h}\text{~~where~~}L_h = \max_{i \in \{1, \dots, n\}} L_i.
\end{equation*} 
However, computing this bound requires the smallest eigenvalue of the mixing matrix $W$ and the global Lipschitz constants of the gradients $\nabla h_i$ for each $i \in \{1, \dots, n\}$, two possibly expensive tasks. Moreover, the latter becomes unviable if at least one of the gradients is only locally Lipschitz continuous.

To overcome the aforementioned challenge {related to the Lipschitz continuity assumptions, assuming that only $\lambda_{\min}(W)$ is available}, we propose a distributed first-order algorithm that uses a linesearch to compute appropriate stepsizes using rounds of peer-to-peer communication. For a sequence of stepsizes $(\tau_k) \subseteq (0, +\infty)$ computed via this linesearch, and a constant $\beta >0$, the proposed algorithm takes the following form: for all $k \geq 1$ and $i \in \{1,\dots,n\}$,
\begin{equation} \label{ADT-i}
    \left\{\begin{aligned}
         u^k_i& = u^{k-1}_i+ \dfrac{\tau_{k-1}}{2}\left( x_i^k - \displaystyle\sum_{j \in N(i)} W_{ij} x_j^k \right)  \\
         \overline{u}^k_i &= u^k_i + \dfrac{\tau_k}{\tau_{k-1}}(u^k_i - u^{k-1}_i)\\
         x^{k+1}_i & = \prox_{\beta\tau_k  f_i}\big(x^k_i - \beta\tau_k \nabla h_i(x^k_i) - \beta\tau_k\overline{u}^k_i\big).
    \end{aligned}\right.
\end{equation} 
The relation between PG-EXTRA and the proposed algorithm is obtained by setting $\tau_k=\tau$, $\beta\tau^2=1$ and $w^k_i = x^k_i - \beta\tau_k \nabla h_i(x^k_i) - \beta\tau_k\overline{u}^k_i$ for each $i = 1, \dots, n$. This change of variables yields~\eqref{agent-PGEXTRA} with $\sigma=\beta\tau$ (see Remark~\ref{remark:PG-E} for further details). As we will see later, the alternative form used in \eqref{ADT-i} will be more convenient than \eqref{agent-PGEXTRA} for the linesearch formulation and analysis.

Our contribution can be summarised as follows. We propose two linesearch procedures for a distributed proximal splitting method to solve the multi-agent minimisation problem \eqref{primal-problem}. In other words, we propose backtracking linesearch routines for PG-EXTRA{, to lift the requirement of knowing the constants of global Lipschitz continuity of gradients, and also allow gradients to be only locally Lipschitz continuous}. Our derivation uses an abstract linesearch framework to define stepsizes for a primal-dual first-order splitting method for solving min-max convex-concave problems with a linear coupling term. We develop a convergence analysis for such splitting method that only requires locally Lipschitz continuous gradients and, in this case, the linesearch estimates the local Lipschitz constant of such gradients. Furthermore, we extend this linesearch to allow approximation of the norm of the linear coupling term, giving more flexibility to the step length in each iteration. 

The paper is organised as follows. Section~\ref{s:prelim} provides the preliminary concepts and materials used in this work. In Section~\ref{s:splitting-LS}, we propose and analyse convergence of with a unifying linesearch framework for a primal-dual algorithm, which is related to the method with linesearch proposed in \cite{malitsky2018first}. Building on this, Section~\ref{s:distributed} presents  two distributed linesearch procedures to define stepsizes for  a  distributed proximal-gradient algorithm with guarantees of convergence to a solution of \eqref{primal-problem}.
Concluding remarks are given in Section~\ref{s:conclusion}.

\section{Preliminaries}
\label{s:prelim}

\subsection{General definitions and notations}

We assume throughout that $\mathcal{U}$ is a finite-dimensional Hilbert space with inner product $\langle \cdot, \cdot \rangle$ and induced norm $\| \cdot \|$. When $\mathcal{U} = \R^d$, the notation $\langle \cdot, \cdot \rangle$ represents the dot product given by $\langle u, v \rangle = u^{\top}v$ for all $u, v \in \mathcal{U}${, where $\top$ denotes the transpose operator of a matrix}. Similarly, when $\mathcal{U} = \R^{n\times d}$, $\langle \cdot, \cdot \rangle$ denotes the trace inner product given by $\langle A, B \rangle = \text{tr}(A^{\top}B)$ for all $A, B \in \mathcal{U}$. {For a matrix $A \in \R^{n \times d}$, the range of $A$ is given by $\range A = \{ y \in \R^n: \exists x \in \R^d, y = Ax \}$.  }For a symmetric matrix $A{\in \R^{n \times n}}$, $\lambda_{\min}(A)$ and $\lambda_{\max}(A)$ denote the smallest and largest eigenvalues of $A$, respectively. {We use $e=(1,\dots,1)^\top$ to denote the vectors of ones of appropriate dimension, and $\oslash$ to denote the Hadamard division of matrices, that is, the element-wise division.}

{Given a linear operator $K: \mathcal{W} \to \mathcal{V}$ between two finite-dimensional Hilbert spaces $\mathcal{W}$ and $\mathcal{V}$ with inner products $\langle \cdot, \cdot\rangle_{\mathcal{W}}$ and $\langle \cdot, \cdot\rangle_{\mathcal{V}}$, respectively, $K^*: \mathcal{V} \to \mathcal{W}$ denotes the adjoint of the operator of $K$, that is, the linear operator that satisfies $$\forall (w,v) \in \mathcal{W}\times \mathcal{V} , \: \langle Kw, v \rangle_{\mathcal{V}} = \langle w, K^*v \rangle_{\mathcal{W}}.$$}

Let $\phi: \mathcal{U}\to \R \cup \{+\infty\}$. The \emph{domain} of $\phi$ is $\dom \phi :=\{x\in \mathcal{U}: \phi(x) <+\infty\}${, and $\overline{\dom \phi}$ denotes its closure}. We say that $\phi$ is \emph{proper} if its domain is nonempty, \emph{lower semicontinuous (lsc)} if, for all $x\in \mathcal{U}$, $\phi(x)\leq \liminf_{y\to x} \phi(y)$, and \emph{convex} if, for all $x, y\in \dom \phi$ and all $\lambda \in (0, 1)$,
\begin{equation*}
\phi(\lambda x +(1 -\lambda)y)\leq \lambda \phi(x) +(1 -\lambda)\phi(y).    
\end{equation*}{The convex conjugate of $\varphi$ is the function $\phi^*: \mathcal{U}\to \R \cup \{+\infty\}$ given by \begin{equation*}
    \phi^*(y) = \sup_{x\in\R^d}\{\langle y,x \rangle - \phi(x)\}.
\end{equation*}For any proper lsc convex function $\phi: \mathcal{U}\to \R \cup \{+\infty\}$, it holds that $\phi = \phi^{**}$. Given $x \in \R^d$, we use $\partial \phi (x)$ to denote the subdifferential of $\phi$ at $x$,  the set of vectors $v \in \R^d$ such that for all $y \in \R^d$, $\phi(y) \geq \phi(x) + \langle v, y - x \rangle.$ The domain of the subdifferential of $\phi$ is the set $\dom \partial \phi = \{x \in \R^d: \partial \phi (x) \neq \emptyset\}$. In particular, for any proper lsc convex function $\phi: \mathcal{U}\to \R \cup \{+\infty\}$, $\dom \partial \phi \neq \emptyset$.}

We say an operator $\Phi: C \subseteq \mathcal{U} \to \mathcal{U}$ is \emph{(globally) Lipschitz continuous} on $C $ with constant $L>0$ if, for all $x,y \in C$,
\begin{equation*}
\|\Phi(x) - \Phi(y)\| \leq L \| x- y\|.    
\end{equation*}
When $C = \mathcal{U}$, we simply say that $\Phi$ is \emph{Lipschitz continuous}.

Functions with Lipschitz gradients provide a powerful ingredient for convergence analysis of gradient-based methods, known as the descent lemma \cite[Lemma 2.64]{bauschke2011convex}. This result states that the first-order approximation error can be bounded by a quadratic term when gradients are Lipschitz continuous. 

\begin{lemma}[Descent lemma] 
\label{descent-lemma}
Suppose $\phi: C \to \R$ is a differentiable function over an open convex set $C \subseteq \mathcal{U}$, such that $\nabla \phi$ is Lipschitz continuous on $C$ with constant $L>0$. Then, for all $x,y \in C$,
    \begin{equation*} 
        |\phi(y) - \phi(x) - \langle \nabla \phi(x), y - x \rangle| \leq \dfrac{L}{2}\| y - x \|^2.
    \end{equation*} 
\end{lemma}

We say that an operator $\Phi: C \subseteq \mathcal{U} \to \mathcal{U}$ is \emph{locally Lipschitz continuous} on $C$ if, for all $x \in \mathcal{U}$, there exists an open set $U_x \subseteq C$ containing $x$, such that $\Phi$ restricted to $U_x$ is Lipschitz continuous. Any twice continuously differentiable function is locally Lipschitz, since they have locally bounded Hessians. For locally Lipschitz functions in finite-dimensional spaces,  Lemma~\ref{descent-lemma} is therefore satisfied over bounded sets.

{
One important application in image processing in which local Lipschitz continuity of the gradients rises naturally is the family of Poisson inverse problems \cite{bertero2009image,bauschke2017descent,li2022splitting}. In particular, for this family, an optimisation problem is solved involving the Kullback–Leibler (KL) divergence of $z \in \R^p$ from $y \in \R^p$: \begin{equation} \label{e:KL}
    \text{KL}(z,y) = \sum_{j=1}^p y_j \ln\left(\dfrac{y_j}{z_j}\right) + z_j - y_j . 
\end{equation} Clearly, the gradient of the KL divergence with respect to $z$, \begin{equation*}
    \nabla_z\text{KL}(z,y) = e - y\oslash z,
\end{equation*} is not globally Lipschitz continuous on $(0,+\infty)^p$, but only locally. In Section~\ref{s:numerical}, we explain how a maximum likelihood approach to model the inverse problem prompts the KL divergence to appear.
}

Given a proper lsc convex function $\phi: \mathcal{U} \to \R \cup \{+\infty\}$ and $\tau>{0}$, the \emph{proximity operator} of $\phi$ with stepsize $\tau >0$ at $x\in\mathcal{U}$ is given by
\begin{equation*}
    \prox_{\tau \phi}(x) = \argmin_{y\in\mathcal{U}} \left\{ \phi(y) + \dfrac{1}{2\tau}\|y-x\|^2 \right\}.
\end{equation*}
For a nonempty closed convex set $X \subseteq \mathcal{U}$, the \emph{indicator function} $\iota_X$ of $X$ given by $\iota_X(x) = 0$ if $x \in X$ and $+\infty$ otherwise, and the \emph{projector} onto $X$ is denoted by $\proj_X :=\prox_{\iota_X}$.

The following lemma is concerned with continuity properties of proximity operators. Here, the notation $\overline{X}$ denotes the \emph{closure} of a set $X\subseteq\mathcal{U}$.

\begin{lemma}[\hspace{-0.09ex}{\cite[Corollary~7]{friedlander2023perspective}}] \label{lemma:prox-extension}
    Let $\phi: \mathcal{U} \to \R \cup \{+\infty\}$ be a proper lsc convex function. Then the operator given by $$[0,+\infty)\times\mathcal{U}\to\mathcal{U}\colon(t, x) \mapsto\begin{cases}
            \prox_{t \phi}(x) & \text{if~} t>0, \\
            \proj_{\overline{\dom(\phi)}}(x) & \text{if~} t=0
    \end{cases}$$
    is continuous on its domain. Moreover, it is locally Lipschitz on the interior of its domain.
\end{lemma}

\subsection{Distributed optimisation}

We now turn our attention to notation and preliminaries for distributed optimisation. Following~\cite{shi2015extra,shi2015proximal,ryu2022large}, given $n$ (column) vectors $x_1,\dots,x_n \in \R^d$, we denote 
\begin{equation*}
\bx =  \begin{pmatrix}
     x_1^\top \\ \vdots \\ x_n^\top \\
 \end{pmatrix} = \begin{pmatrix}
     x_1 & \cdots & x_n \\
 \end{pmatrix}^\top \in\mathbb{R}^{n\times d} 
 \end{equation*}
For the functions in problem \eqref{primal-problem}, we define $f : \R^{n \times d} \to \R \cup \{+\infty\}$ and $h : \R^{n \times d} \to \R$ as 
\begin{equation*}
     f(\bx) := \sum_{i = 1}^n f_i(x_i) \text{~~and~~} h(\bx) := \sum_{i = 1}^n h_i(x_i).
\end{equation*} 
Under this definition, the proximity operator of $f$ with parameter $\tau$ and the gradient of $h$ are given by
 \begin{equation*}
   \prox_{\tau f}(\bx) = \begin{pmatrix} \prox_{\tau f_1}(x_1)^\top \\ \vdots \\ \prox_{\tau f_n}(x_n)^\top \end{pmatrix}\text{~~and~~}
     \nabla h(\bx) = \begin{pmatrix}
     \nabla h_1(x_1)^\top \\ \vdots \\ \nabla h_n(x_n)^\top \\
 \end{pmatrix}.
 \end{equation*}

In order to model communication between agents through the communication graph $(V,E)$ in~\eqref{primal-problem}, we recall the following definition (see, for example, \cite{shi2015extra}).

\begin{definition}[Mixing matrix] 
\label{def:mixing-matrix}
Given an undirected graph $(V,E)$ with $|V|=n$, a matrix $W \in \R^{n \times n}$ is called a \emph{mixing matrix} if
\begin{enumerate}
\item\label{def:mixing-matrix_nei}
$W_{ij} = 0$ whenever $ij \notin E$,
\item\label{def:mixing-matrix_sym}
$W = W^{\top}$,
\item\label{def:mixing-matrix_cons}
$\ker(I - W) = \R e$, where $e=(1,\dots,1)^\top$,
\item\label{def:mixing-matrix_eig}
$-I \prec W \preceq I$.
\end{enumerate}
\end{definition}

Definition~\ref{def:mixing-matrix}\ref{def:mixing-matrix_nei}--\ref{def:mixing-matrix_sym} ensures that communication is symmetric and only occurs between neighbours, while Definition~\ref{def:mixing-matrix}\ref{def:mixing-matrix_cons} is used to ensure consensus, and Definition~\ref{def:mixing-matrix}\ref{def:mixing-matrix_eig} is related to convergence. Note that, under the notation defined in this section,  \eqref{agent-PGEXTRA} can be combined for all agents $i\in\{1,\dots,n\}$ to write the PG-EXTRA compactly as
\begin{equation}\label{PG-EXTRA}
    \left\{\begin{aligned}
         \bw^k& = \bw^{k-1} + W\bx^k   - \dfrac{1}{2}(W+I)\bx^{k-1} - \tau\left( \nabla h(\bx^k)-\nabla h(\bx^{k-1})\right)  \\
         \bx^{k+1} & = \prox_{\tau  f}\big(\bw^k\big).
    \end{aligned}\right.
\end{equation} 

As discussed in \cite{malitsky2023first}, PG-EXTRA can be understood as the \emph{{Condat--V\~{u}} method}~\cite{condat2013primal,vu2013splitting} applied to a reformulation of \eqref{primal-problem} as a saddle-point problem (see also Section~\ref{s:distributed}). In what follows, we recall the derivation of this reformulation. For each agent $i\in\{1,\dots,n\}$, let $x_i\in\mathbb{R}^d$ denote its copy of the decision variable $x\in\mathbb{R}^d$. Then problem~\eqref{primal-problem} can be expressed as
\begin{equation}\label{primal-problem extended} 
    \min_{x_1,\dots,x_n \in \R^d} \sum_{i = 1}^n f_i(x_i) + h_i(x_i)\text{~~s.t.~~}  x_1 = \dots = x_n.
\end{equation}
Let $U := \bigl((I-W)/2\big)^{1/2}$ so that $\ker(U) = \ker(U^2)=\ker(I-W)=\mathbb{R}e$. Then $x_1 = \dots = x_n $ if and only if $U \bx = 0$, in which case we say $\bx$ is in \emph{consensus}. Hence, problem \eqref{primal-problem extended} becomes
\begin{equation} \label{primal-problem-2}
    \min_{\bx \in \R^{n \times d}} \iota_{\{0\}}(U\bx) + f(\bx) + h(\bx). 
\end{equation}
Since $\iota_{\{0\}}$ is proper lsc convex with $\iota^*_{\{0\}}=0$, we have $\iota_{\{0\}}=\iota_{\{0\}}^{**}=\sup_{\by\in\mathbb{R}^{n\times d}}\langle \cdot ,\by\rangle$. Substituting this into \eqref{primal-problem-2} therefore gives
\begin{equation}\label{eq:extra saddle}  
\min_{\bx \in \R^{n \times d}}\sup_{\by \in \R^{n \times d}}  \langle U\by,\bx \rangle + f(\bx) + h(\bx).
\end{equation} 
Altogether, this shows that the multi-agent minimisation problem \eqref{primal-problem} can be reformulated as a saddle-point problem~\eqref{eq:extra saddle} having a linear coupling term related to consensus among agents. Note that, in this context, $\|U\|$ can be computed from the mixing matrix $W$. The structure of the latter problem~\eqref{eq:extra saddle}  motivates our study of min-max problems in the following section.

\section{Proximal splitting algorithm with linesearch} \label{s:splitting-LS}
Consider the following convex-concave  min-max problem
\begin{equation} \label{MP-problem}
    \min_{u \in \mathcal{U}} \max_{v \in \mathcal{V}} \: \Psi(u,v):=\langle Ku,v\rangle + G(u) - F(v) - H(v),
\end{equation}
where $K: \mathcal{U} \to \mathcal{V}$ is a  (bounded) linear operator between two finite-dimensional Hilbert spaces, $G: \mathcal{U} \to \R \cup \{+\infty\} $ and $F: \mathcal{V} \to \R \cup \{+\infty\}$ are proper lsc convex functions, and $H: \mathcal{V} \to \R$ is convex differentiable with locally Lipschitz continuous gradient. We assume that $\|K\|$ is known and that there exists a \emph{saddle point} for problem \eqref{MP-problem}, that is, there exists a pair $(\hat{u}, \hat{v}) \in \mathcal{U} \times \mathcal{V}$ satisfying
\begin{equation}\label{eq:saddle defn}
\forall (u,v)\in\mathcal{U}\times\mathcal{V},\quad \Psi(\hat{u},v)\leq \Psi(\hat{u},\hat{v})\leq \Psi(u,\hat{v}).
\end{equation}
Conditions for existence of saddle points for convex-concave problems can be found, for example, in \cite{rockafellar1970monotone,rockafellar1971saddle}.  Note that \eqref{eq:saddle defn} can be equivalently expressed as the following pair of inequalities
\begin{equation} \label{def:saddle-point-2}
    \left\{\begin{aligned}
    \forall u \in \mathcal{U}, &&\hat{P}_{(\hat{u},\hat{v})}(u) &:= G(u) - G(\hat{u}) + \langle K^* \hat{v}, u - \hat{u} \rangle\geq0, \\
    \forall v \in \mathcal{V}, &&\hat{D}_{(\hat{u},\hat{v})}(v) &:= (F+H)(v) - (F+H)(\hat{v}) - \langle K \hat{u}, v - \hat{v} \rangle\geq0.
    \end{aligned}\right.
\end{equation}
The two quantities above define the \emph{primal-dual gap}. Given a saddle point $(\hat{u}, \hat{v})$ of problem \eqref{MP-problem}, the \emph{primal-dual gap} is given by \begin{equation} \label{primal-dual-gap}
    \hat{\mathcal{G}}_{(\hat{u},\hat{v})}(u,v) = \hat{P}_{(\hat{u},\hat{v})}(u) + \hat{D}_{(\hat{u},\hat{v})}(v)\geq 0.
\end{equation} 

Assuming $\nabla H$ is globally Lipschitz continuous (with constant $L>0$) and knowledge of~$\|K\|$, the \emph{Condat--V\~u method}~\cite{condat2013primal,vu2013splitting} can be used to solve~\eqref{MP-problem}. Given stepsizes $\tau,\sigma>0$ satisfying $\frac{L\sigma}{2}+\tau\sigma\|K\|^2 \leq 1$, this method takes the form
\begin{equation*} 
    \left\{\begin{aligned}
        u^k & =   \prox_{\tau G}\big(u^{k-1} - \tau K^*v^k\big) \\
        \overline{u}^k & =   2u^k - u^{k-1} \\
        v^{k+1} & =   \prox_{\sigma F}\big( v^k + \sigma \big[K \overline{u}^k - \nabla H(v^k) \big] \big).
    \end{aligned}\right.
\end{equation*}
Note that, in the special case where $H=0$, this method reduces to the \emph{primal-dual hybrid gradient (PDHG) method} from \cite{chambolle2011first}.  To handle the case where useful estimates of both $L$ and $\|K\|$ are unavailable, a version of the {Condat--V\~{u}} method with a backtracking linesearch was developed in \cite[Section~4]{malitsky2018first}. However, this approach is not directly suited to our setting for two reasons: firstly, it still requires global Lipschitz continuity of $H$ (even though the Lipschitz constant is not needed) and, secondly, it can not exploit knowledge of $\|K\|$ even when it is available. In the distributed setting of Section~\ref{s:distributed}, the latter is related to the minimum eigenvalue of the mixing matrix and, as such, {often needs to be known to determine an interval for valid stepsizes that guarantee convergence (see, e.g., \cite{shi2015extra,shi2015proximal,li2019decentralized,alghunaim2020decentralized,malitsky2023first}). When this quantity is unknown, an approximation can be used based on the maximum degree of the network, as explained in \cite[Section~2.4(ii)]{shi2015extra}.}   


Thus, in order to solve problem \eqref{MP-problem} in the setting important for distributed optimisation, in this section, we extend the analysis of the primal-dual algorithm with linesearch introduced in \cite{malitsky2018first} under the assumption that $\nabla H$ is locally Lipschitz (but not necessarily globally Lipschitz) and that $\|K\|$ is known.  Our analysis considers an abstract linesearch procedure satisfying Assumptions~\ref{assumption-LS} and~\ref{assumption-LS-iii} (for use in Section~\ref{s:distributed}).

\subsection{Primal-dual algorithm with abstract linesearch}

 In Algorithm~\ref{a:P-D}, we present our proposed method with an abstract linesearch procedure, denoted \texttt{LS}, which is assumed to satisfy Assumption~\ref{assumption-LS}. The main result of this section (Theorem~\ref{th:abstract-convergence}) shows that the sequences generated by Algorithm~\ref{a:P-D} converge to a saddle point of \eqref{MP-problem}. An implementation of  the abstract linesearch is given in Subroutine~\ref{LS:0}.

\begin{assumption}[Abstract linesearch] \label{assumption-LS}

{Given a set of global parameters $\mathcal{P}$, and the input}
$$(\texttt{param}, \tau_{\text{init}}, \tau, u^-, u, v) \in \mathcal{P} \times  (0,+\infty)\times  (0,+\infty) \times \mathcal{U} \times \mathcal{U} \times \mathcal{V}$$ {the output of the abstract linesearch procedure \texttt{LS},} 
$$(\tau^+, v^+) \in (0,+\infty) \times \mathcal{V},$$ {satisfies} the following conditions.
\begin{enumerate}
\item
$\tau^+ \leq \tau_{\text{init}}$.
        
\item
{for some $\beta>0$ and $\delta_L \in (0,1)$, such that $\beta,\delta_L \in \texttt{param}$, it holds } $$  \tau^+\left( H(v^+) - H(v) - \langle \nabla H(v), v^+ - v \rangle \right) \leq \dfrac{{\delta_L}}{2\beta} \|v^+ - v\|^2.$$
\end{enumerate}
\end{assumption}

The set of parameters \texttt{param} in Assumption~\ref{assumption-LS}, besides including the constants $\beta$ and $\deltaL$ that appear in the descent condition (ii), it may also involve other parameters that define internal operations in \texttt{LS}. We will define a specific set of parameters in Section~\ref{ss:LS-implementation}.

\begin{algorithm}[!htb]
\caption{Primal-dual algorithm with linesearch for problem~\eqref{MP-problem}.\label{a:P-D}}

Choose initial points $u^0 \in \mathcal{U}$, $v^1 \in \mathcal{V}$\;
Choose a set of parameters \texttt{param} containing $\beta >0$ and $ \deltaL \in (0,1)$\;
Choose parameters $\tau_0 >0$, $\gamma \in (0,1),$ and $\deltaB \in (0,1)$ such that $\deltaB + \deltaL < 1$\;
Set $\theta_0=1$\;
\For{$k=1,2,\dots \:$}{

    \step{Dual update} Define\begin{equation*} 
        u^k = \prox_{\tau_{k-1}G}(u^{k-1} - \tau_{k-1}K^* v^k).
    \end{equation*}

     \step{Backtracking linesearch and primal update} Initialise the linesearch by choosing  $\alpha_k \in [1, \sqrt{1 + \gamma\theta_{k-1}}]$, and define 
    \begin{equation} \label{init-tau0}\tau_{k(0)} = \min\left\{\dfrac{\sqrt{\deltaB}}{\sqrt{\beta}\|K\|}, \tau_{k-1}\alpha_k\right\}.
    \end{equation}

    Compute \begin{equation*}
        {(}\tau_k, v^{k+1}{)} = \texttt{LS}(\texttt{param}, \tau_{k(0)}, \tau_{k-1}, u^{k-1}, u^k, v^k),
    \end{equation*} 
    
    where \begin{align*}
    \theta_k &= \tau_{k}\tau_{k-1}^{-1} \\
    \overline{u}^{k}  &= u^k + \theta_k(u^k - u^{k-1}) \\
        v^{k+1} &= \prox_{\beta \tau_{k} F}\big( v^k + \beta \tau_{k} [K\overline{u}^{k} - \nabla H(v^k)]\big).
    \end{align*}
}
\end{algorithm}

 In order to prove convergence of Algorithm~\ref{a:P-D}, motivated by \cite[Theorem 3.4]{malitsky2018first}, we first show that the algorithm satisfies a sufficient descent condition and generates bounded sequences. 

\begin{lemma}[Sufficient decrease and boundedness of iterates] 
\label{prop:bounded}
 Suppose that $F: \mathcal{V} \to \R \cup \{+\infty\}$ and $G: \mathcal{U} \to \R \cup \{+\infty\}$ are proper lsc convex functions, and $H: \mathcal{V} \to \R$ is a differentiable convex function with locally Lipschitz continuous gradient. 
In addition, suppose Assumption~\ref{assumption-LS} holds, and let  $(\hat{u},\hat{v})$ be a saddle point of \eqref{MP-problem}.  Then, for the sequences $(v^k)$, $(\overline{u}^k)$ and $(u^k)$ generated by Algorithm~\ref{a:P-D}, the following hold: 
\begin{enumerate}
\item
The sequence $(\varphi_k)$ defined by \begin{equation} \label{def:varphik}\varphi_k:=\tau_{k-1}(1 + \theta_{k-1}) \hat{P}_{(\hat{u}, \hat{v})}(u^{k-1}) +  \dfrac{1}{2} \| u^k - \hat{u}\|^2   + \dfrac{1}{2\beta} \|v^k - \hat{v} \|^2,\end{equation} satisfies the sufficient decrease condition
\begin{multline}\label{eq:sufficient decrease}
    0\leq \varphi_{k+1} \leq \varphi_k  - \dfrac{1}{2}    \|\overline{u}^k - u^k\|^2  - \dfrac{1-(\deltaB +\deltaL)}{2\beta} \|  v^{k+1} -  v^k\|^2 \\ - \tau_k\hat{D}_{(\hat{u}, \hat{v})}(v^{k+1}) - (1-\gamma)\tau_{k-1}\theta_{k-1}\hat{P}_{(\hat{u}, \hat{v})}(u^{k-1}). 
\end{multline} 

\item
The sequences $(v^k)$, $(\overline{u}^k)$ and $(u^k)$   are bounded, and 

\item
$(\|v^{k+1} - v^k\|)$, $(\|\overline{u}^k - u^k\|)$, $(\tau_{k}\theta_{k}\hat{P}_{(\hat{u}, \hat{v})}(u^{k}))$, and $(\tau_k\hat{D}_{(\hat{u}, \hat{v})}(v^{k+1}))$ converge to $0$.
\end{enumerate}
\end{lemma}

\begin{proof}
In the context of Algorithm~\ref{a:P-D}, Assumption~\ref{assumption-LS}(ii) implies:  \begin{equation} \label{LS-term}
\forall k\in\mathbb{N}, \quad \tau_{k}\left( H(v^{k+1}) - H(v^k) - \langle \nabla H(v^k), v^{k+1} - v^k \rangle \right) \leq \dfrac{\deltaL}{2\beta} \|v^{k+1} - v^k\|^2.     \end{equation}In view of the update rule of $v^{k+1}$ in Algorithm~\ref{a:P-D} and \cite[Proposition~12.26]{bauschke2011convex}, we have
\begin{equation}\label{eq:prox F}
\forall v \in \mathcal{V},\: F(v^{k+1}) + \left\langle \dfrac{1}{\beta \tau_{k}}(v^k-v^{k+1}) + K\overline{u}^k, v - v^{k+1}\right\rangle - \langle \nabla H(v^k), v - v^{k+1} \rangle \leq F(v).
\end{equation} In view of Assumption~\ref{assumption-LS}({i}) {and \eqref{init-tau0}}, $\tau_k^2\|K\|^2\leq \frac{\deltaB}{\beta}$,  and thus
$\tau_k^2\|K^* v^{k+1} - K^* v^k\|^2 \leq \frac{\deltaB}{\beta}\|v^{k+1}-v^k\|^2. $
This estimate together with  condition \eqref{LS-term} and the convexity of $H$ implies
\begin{equation}\label{eq:desc H}
\begin{aligned}
 H(v^{k+1}) 
 &\leq H(v^k) + \langle \nabla H(v^k), v^{k+1} - v^k \rangle +\dfrac{\deltaL}{2\beta\tau_k} \|v^{k+1} - v^k\|^2 \\
 &\leq H(v^k) + \langle \nabla H(v^k), v- v^k \rangle  + \langle \nabla H(v^k), v^{k+1} - v\rangle + \dfrac{\deltaL}{2\beta\tau_k} \|v^{k+1} - v^k\|^2 \\
 &\quad +\dfrac{\deltaB}{2\beta\tau_k} \|v^{k+1} - v^k\|^2  -\frac{\tau_{k}}{2} \|K^* v^{k+1} - K^* v^k\|^2 \\
 &\leq H(v) + \langle \nabla H(v^k), v^{k+1} - v\rangle  + \dfrac{\deltaB+\deltaL}{2\beta\tau_k} \|v^{k+1} - v^k\|^2  - \frac{\tau_{k}}{2} \|K^* v^{k+1} - K^* v^k\|^2.
\end{aligned}
\end{equation}
Thus, combining \eqref{eq:prox F} and \eqref{eq:desc H} yields
\begin{multline} \label{eq:12}
            \tau_k\bigl((F+H)(v^{k+1})-(F+H)(v)\bigr) \\
            \leq \left\langle \dfrac{1}{\beta}(v^{k+1}-v^k) - \tau_k K\overline{u}^k, v - v^{k+1}\right\rangle + \dfrac{\deltaB+\deltaL}{2\beta}\|v^{k+1}-v^k\|^2 - \frac{\tau_{k}^2}{2}\|K^* v^{k+1} - K^* v^k\|^2
\end{multline} for all $v \in \mathcal{V}$. With \eqref{eq:12} established, the remainder of the proof closely follows  \cite[Theorem 3.4]{malitsky2018first}. From the update rule for $u^{k+1}$ in Algorithm~\ref{a:P-D} and \cite[Proposition~12.26]{bauschke2011convex}, we have
\begin{equation} \label{G-ineq}
\forall u \in \mathcal{U}, \: \tau_{k}(G(u^{k+1})-G(u)) \leq \langle u^{k+1}- u^{k}  + \tau_{k}K^*v^{k+1}, u - u^{k+1}\rangle.
 \end{equation}
Using \eqref{G-ineq}, the identity $\tau_k=\theta_k\tau_{k-1}$ and the definition of ${\overline{u}}^k$, we deduce that
\begin{equation}\label{eq:all the Gs}\begin{aligned}
& \tau_k\bigr( (1+\theta_k)G(u^k) - \theta_kG(u^{k-1}) -  G(u)  \big) \\
&= \theta_k \tau_{k-1}\bigr(G(u^k)-G(u^{k+1})\bigr) + \theta_k^2\tau_{k-1}\bigl(G(u^k)-G(u^{k-1})\bigr) + \tau_k\bigl(G(u^{k+1})-G(u)\bigr) \\
&= \theta_k\langle u^{k}- u^{k-1}  + \tau_{k-1}K^*v^{k}, u^{k+1} - u^{k}\rangle
+ \theta_k^2\langle u^{k}- u^{k-1}  + \tau_{k-1}K^*v^{k}, u^{k-1} - u^{k}\rangle\\
&\qquad 
+ \langle u^{k+1}- u^{k}  + \tau_{k}K^*v^{k+1}, u - u^{k+1}\rangle \\
&= \langle {\overline{u}}^{k}- u^k  + \tau_{k}K^*v^{k},u^{k+1} - {\overline{u}}^{k}\rangle + \langle u^{k+1}- u^{k}  + \tau_{k}K^*v^{k+1}, u - u^{k+1}\rangle. 
\end{aligned}\end{equation}
Adding \eqref{eq:12} and \eqref{eq:all the Gs} yields
\begin{multline}\label{eq:pre key}
 \tau_k\big((F+H)(v^{k+1}) - (F+H)(v)   + (1+\theta_k)G(u^k) - \theta_kG(u^{k-1}) -  G(u)  \big) \\
\leq \langle u^{k+1} - u^{k}, u - u^{k+1}\rangle + \dfrac{1}{\beta}\left\langle v^{k+1} - v^k, v - v^{k+1}\right\rangle + \langle \overline{u}^k - u^k, u^{k+1} - \overline{u}^k \rangle 
+ \dfrac{\deltaB+\deltaL}{2\beta}\|v^{k+1}-v^k\|^2\\
- \frac{\tau_{k}^2}{2}\|K^* v^{k+1} - K^* v^k\|^2  +\tau_k\left(\langle   K^*v^{k+1}, u - u^{k+1}\rangle - \left\langle   K\overline{u}^k, {v} - v^{k+1}\right\rangle  + \langle  K^*v^k, u^{k+1} - \overline{u}^k \rangle\right)  ,
\end{multline} for all $(u,v) \in \mathcal{U} \times \mathcal{V}$. Using properties of the adjoint, the inner product terms on last line of \eqref{eq:pre key} can be written as
\begin{equation}\label{eq:adjoints}
\begin{aligned}
&\tau_k \langle  K^*v^k, u^{k+1} - \overline{u}^k \rangle + \tau_k\langle   K^*v^{k+1}, u - u^{k+1}\rangle -\tau_k \left\langle   K\overline{u}^k, v - v^{k+1}\right\rangle  \\
&= \tau_k \langle  K^*v^{k+1}-K^*v^k, \overline{u}^k-u^{k+1} \rangle + \tau_k\langle   K^*v^{k+1}, u - {\overline{u}}^{k}\rangle -\tau_k \left\langle   K\overline{u}^k, v - v^{k+1}\right\rangle \\
&= \tau_k \langle  K^*v^{k+1}-K^*v^k, \overline{u}^k-u^{k+1} \rangle + \tau_k\langle   v^{k+1}-v, Ku  \rangle -\tau_k \left\langle   K\overline{u}^k-Ku, v \right\rangle\\
&= \tau_k \langle  K^*v^{k+1}-K^*v^k, \overline{u}^k-u^{k+1} \rangle + \tau_k\langle   v^{k+1}-v, Ku  \rangle
-\tau_k(1+\theta_k) \left\langle  u^k-u, K^*v \right\rangle\\
&\qquad +\tau_k\theta_k \left\langle   u^{k-1}-u, K^*{v} \right\rangle.
\end{aligned}
\end{equation}

Let $(\hat{u},\hat{v})$ be a saddle point of problem \eqref{MP-problem}, and take $(u,v) = (\hat{u},\hat{v})$. Substituting \eqref{eq:adjoints} into \eqref{eq:pre key}, noting the definitions for $\hat{P}_{(\hat{u},\hat{v})}$ and $\hat{D}_{(\hat{u},\hat{v})}$ from \eqref{primal-dual-gap}, we obtain
\begin{multline*} 
\tau_k \big( \hat{D}_{(\hat{u}, \hat{v})}(v^{k+1}) - \theta_k\hat{P}_{(\hat{u}, \hat{v})}(u^{k-1}) + (1 + \theta_k)\hat{P}_{(\hat{u}, \hat{v})}(u^k) \big)  \\
\leq \langle u^{k+1} - u^{k}, \hat{u} - u^{k+1}\rangle + \dfrac{1}{\beta}\left\langle v^{k+1} - v^k, \hat{v} - v^{k+1}\right\rangle + \langle \overline{u}^k - u^k, u^{k+1} - \overline{u}^k \rangle \\
\quad + \dfrac{\deltaB+\deltaL}{2\beta}\|v^{k+1}-v^k\|^2  + \tau_k \langle  K^*v^{k+1}-K^*v^k, \overline{u}^k-u^{k+1} \rangle- \frac{\tau_{k}^2}{2}\|K^* v^{k+1} - K^* v^k\|^2.
\end{multline*}
Using $\langle a,b \rangle \leq \frac{1}{2}\|a\|^2 + \frac{1}{2}\|b\|^2$ in the last line and the law of cosines in the second line yields \begin{multline*} 
\tau_k \big( \hat{D}_{(\hat{u}, \hat{v})}(v^{k+1}) - \theta_k\hat{P}_{(\hat{u}, \hat{v})}(u^{k-1}) + (1 + \theta_k)\hat{P}_{(\hat{u}, \hat{v})}(u^k) \big)  \\
\leq \dfrac{1}{2}\left( \| u^k - \hat{u}\|^2 - \| u^{k+1} - u^{k} \|^2 - \|u^{k+1} - \hat{u} \|^2  \right) 
+ \dfrac{1}{2\beta}\left( \|v^k - \hat{v} \|^2 - \|v^{k+1} - v^k\|^2 - \| v^{k+1} - \hat{v} \|^2 \right)\\
\quad + \dfrac{1}{2} \left( \| u^{k+1} - u^k \|^2 - 
             \|\overline{u}^k - u^k\|^2 - \|u^{k+1} - \overline{u}^k\|^2 \right) 
   +\dfrac{{\deltaB+\deltaL}}{{2\beta}} \|  v^{k+1} -  v^k\|^2 + \dfrac{1}{2} \| \overline{u}^k - u^{k+1} \|^2 \\
= \dfrac{1}{2}\left( \| u^k - \hat{u}\|^2  - \| u^{k+1} - \hat{u} \|^2  \right) + \dfrac{1}{2\beta}\left( \|v^k - \hat{v} \|^2  - \| v^{k+1} - \hat{v} \|^2 \right) \\ - \dfrac{1}{2}    
            \|\overline{u}^k - u^k\|^2   - \dfrac{1-(\deltaB+\deltaL)}{{2\beta}} \|  v^{k+1} -  v^k\|^2. 
\end{multline*}
From \eqref{init-tau0} and Assumption~\ref{assumption-LS}(i), $\tau_k \leq \tau_{k-1}\alpha_k$, and {$\theta_k = \tau_k \tau_{k-1}^{-1}$ and $\alpha_k \leq \sqrt{1+\gamma \theta_{k-1}}$}, thus $\theta_k \tau_k \leq \tau_{k-1}(1 + \gamma\theta_{k-1})$. Then, $(\varphi_k)$ satisfies the sufficient decrease condition~\eqref{eq:sufficient decrease}.  Hence, $(\varphi_k)$ converges to $\inf_k \varphi_k \geq 0$. The sufficient decrease condition implies  \begin{equation*} \|\overline{u}^k - u^k\|^2 , \: \|  v^{k+1} -  v^k\|^2, \: \tau_k\hat{D}_{(\hat{u}, \hat{v})}(v^{k+1}), \: \tau_{k-1}\theta_{k-1}\hat{P}_{(\hat{u}, \hat{v})}(u^{k-1}) \to 0 \text{ as } k \to +\infty.
\end{equation*}  Furthermore,  since $\hat{P}_{(\hat{u}, \hat{v})}(u^{k-1}) \geq 0$ and $(\varphi_k)$ converges, the sequences $(u^k)$ and $(v^k)$ are bounded, which in turn implies that $(\overline{u}^k)$ is also bounded.
\end{proof}

Global convergence will follow from the previous results, as long as the sequence of stepsizes $(\tau_k)$ does not vanish. This fact determines the next assumption of the abstract linesearch.

\begin{assumption} \label{assumption-LS-iii}
    {Suppose the sequence $(\tau_k)$ generated by Algorithm~\ref{a:P-D} implemented with \texttt{LS} is separated from $0$.}
\end{assumption}

{Assumption~\ref{assumption-LS-iii} is key in ensuring convergence of the iterates. Observe that in the case when the gradient of $H$ is globally Lipschitz continuous, Assumption~\ref{assumption-LS-iii} is satisfied, which can be shown following a similar argument to \cite[Lemma 3.3 (ii)]{malitsky2018first}. Naturally, Assumption~\ref{assumption-LS-iii} may not hold necessarily when $\nabla H$ is locally Lipschitz only. Nevertheless, we will see in Section~\ref{ss:LS-implementation} that Assumption~\ref{assumption-LS-iii} holds in the latter case for the subroutines defined therein, due to boundedness of the iterates. Furthermore,} note that, as a consequence of \eqref{init-tau0} and Assumptions~\ref{assumption-LS}(i) and~\ref{assumption-LS-iii}, since $\theta_k = \tau_k \tau_{k-1}^{-1}$, the sequence $(\theta_k)$ is bounded and separated from zero.


\begin{theorem}
\label{th:abstract-convergence}
Suppose that $F: \mathcal{V} \to \R \cup \{+\infty\}$ and $G: \mathcal{U} \to \R \cup \{+\infty\}$ are proper lsc convex functions, and $H: \mathcal{V} \to \R$ is a differentiable convex function with locally Lipschitz continuous gradient.  Suppose, in addition, Assumptions~\ref{assumption-LS} and~\ref{assumption-LS-iii} hold, and there exists a saddle point of \eqref{MP-problem}.  Then, the sequence $(u^k, v^k)$ generated by Algorithm~\ref{a:P-D} converges to a saddle point $(\hat{u}, \hat{v})$ of problem \eqref{MP-problem}, and  $\left(\hat{\mathcal{G}}_{(\hat{u}, \hat{v})}(u^k,v^{k})\right)$ converges to $0$. 
\end{theorem}

\begin{proof}

    We first prove that any cluster point of $(u^k, v^k)$ is a saddle point. Let $(\hat{u}, \hat{v})$ be a cluster point of $(u^k, v^k)$ which is bounded due to Lemma~\ref{prop:bounded}(ii). Then there exists a subsequence such that $(u^{\ell_k}, v^{\ell_k}) \to (\hat{u}, \hat{v})$ as $k \to +\infty$.
    From \eqref{eq:12} and \eqref{G-ineq}, we have \begin{equation*}
    \begin{array}{rcl}
       (F+H)(v^{\ell_k+1})-(F+H)(v)&
            \leq & \left\langle \dfrac{1}{\beta \tau_{\ell_k}}(v^{\ell_k+1}-v^{\ell_k}) -  K\overline{u}^{\ell_k}, v - v^{\ell_k+1}\right\rangle + \dfrac{\deltaL}{\beta\tau_{\ell_k}}\|v^{\ell_k+1}-v^{\ell_k}\|^2  \\
         G(u^{\ell_k+1})-G(u)& \leq &\left\langle \dfrac{1}{\tau_{\ell_k}}(u^{\ell_k+1}- u^{\ell_k})  + K^*v^{\ell_k+1}, u - u^{\ell_k+1}\right\rangle,
    \end{array} 
\end{equation*} for all $(u,v) \in \mathcal{U}\times\mathcal{V}$. In view of Assumptions~\ref{assumption-LS}(i) and~\ref{assumption-LS-iii} and Lemma~\ref{prop:bounded}(iii), taking the limit as $k \to +\infty$, yields $\dfrac{1}{\beta \tau_{\ell_k}}(v^{\ell_k+1}-v^{\ell_k}) \to 0$ and $ \dfrac{1}{\tau_{\ell_k}}(u^{\ell_k+1}- u^{\ell_k})  \to 0$, and thus $\hat{P}_{(\hat{u},\hat{v})}(u) , \hat{D}_{(\hat{u},\hat{v})}(v) \geq 0$ for all $(u,v) \in \mathcal{U}\times\mathcal{V}$. That $(\hat{u}, \hat{v})$ is a saddle point for problem \eqref{MP-problem} follows from \eqref{def:saddle-point-2}. Furthermore, substituting $\ell_k$ with $\ell$ in the above set of inequalities and $(u,v) = (\hat{u}, \hat{v})$, taking the limit as $\ell \to +\infty$, since both $(\tau_k)$ and $(\theta_k)$ are separated from zero, Lemma~\ref{prop:bounded}(iii) implies $\hat{D}_{(\hat{u}, \hat{v})}(v^{k}) \to 0$ and $\hat{P}_{(\hat{u}, \hat{v})}(u^{k}) \to 0$, so that $\hat{\mathcal{G}}_{(\hat{u}, \hat{v})}(u^k,v^{k}) \to 0$ follows from~\eqref{primal-dual-gap}.    
Finally,  we prove that the whole sequence $(u^k, v^k)$ converges to $(\hat{u}, \hat{v})$. Since $(\varphi_k)$ defined in \eqref{def:varphik} (monotonically) converges to some $\tilde{\varphi} \geq 0$, then $$\dfrac{1}{2} \| u^k - \hat{u}\|^2   + \dfrac{1}{2\beta} \|v^k - \hat{v} \|^2 \to \tilde{\varphi} \text{ as } k \to +\infty.$$ Taking the limit along the subsequence $(u^{\ell_k}, v^{\ell_k})$ as $k \to +\infty$ in the previous equation, yields $\tilde\varphi=0$. Thus $u^k \to \hat{u}$ and $v^k \to \hat{v}$ as $k \to +\infty$. 
\end{proof}

\begin{remark} \label{r:3.5}

Some comments regarding the proof of the result above are in order.

\begin{enumerate} 
    \item 

    By using with the argument in \cite[Theorem~3.5]{malitsky2018first}, it is possible to derive the following ergodic rate for the primal-dual gap in Theorem~\ref{th:abstract-convergence}:
 $$ \hat{\mathcal{G}}_{(\hat{u}, \hat{v})}(U^K,V^{K}) \leq \frac{1}{s_K}\left(\tau_1\theta_1\hat{P}_{(\hat{u},\hat{v})}(u^0)+\frac{1}{2}\|u^1-\hat{u}\|^2 + \frac{1}{2\beta}\|v^1-\hat{v}\|^2\right), $$
where $s_K=\sum_{k=1}^K\tau_k$, $U^K=\dfrac{\tau_1\theta_0u^0+\sum_{k=1}^K\tau_k{\overline{u}}^k}{\tau_1\theta_1+s_N}$ and $V^K=\dfrac{\sum_{k=1}^K\tau_kv^k}{s_K}$. {Furthermore, note that in view of Assumption~\ref{assumption-LS-iii}, there exists $\tau_{\min} >0$, such that for all $k \geq 1$, $s_k \geq k \tau_{\min}$, and thus  $\hat{\mathcal{G}}_{(\hat{u}, \hat{v})}(U^k,V^{k}) = O(\frac{1}{k})$.}

{\item Similarly to \cite[Section~4]{malitsky2018first}, it is possible to modify Algorithm~\ref{a:P-D} to improve the convergence results in the strongly convex case. If $F$ is $\nu$-strongly convex for $\nu >0$, given $\beta_0 >0$, then defining $$\beta_k = \dfrac{\beta_{k-1}}{1 + \nu \beta_{k-1}\tau_{k-1}}$$ in iteration $k$ of Algorithm~\ref{a:P-D} Step 1, and substituting $\beta$ with $\beta_k$ in \eqref{not-LS-cond}, yields $\|v^{k} - \hat{v}\| = O(\frac{1}{k})$ and $\hat{\mathcal{G}}_{(\hat{u}, \hat{v})}(U^k,V^{k}) = O(\frac{1}{k^2})$ (cf. Remark~\ref{r:3.5}(i)) as in \cite[Theorem 4.2]{malitsky2018first}. 
As for the modification of Algorithm~\ref{a:P-D} when $G$ is strongly convex, we refer to \cite[Section~4.1]{malitsky2018first}. However, the latter variation is not meaningful in our case, since our main focus in Section~\ref{s:distributed} is the distributed setting and the assumption does not hold, see \eqref{eq:identifications}.
}

\item 
The introduction of the parameter $\gamma \in (0,1)$ in Step 2 of Algorithm~\ref{a:P-D} is crucial to guarantee global convergence of the sequence. Otherwise, if $\gamma =1$, we would only be able to prove subsequential convergence to saddle points. Indeed, setting $\gamma=1$ in \eqref{eq:sufficient decrease} would prevent us from concluding that $(\tau_{k}\theta_{k}\hat{P}_{(\hat{u}, \hat{v})}(u^{k}))$ converges to $0$, and thus we would only be able to guarantee that $\left(\dfrac{1}{2} \| u^k - \hat{u}\|^2   + \dfrac{1}{2\beta} \|v^k - \hat{v} \|^2 \right)$ converges (but not necessarily to $0$). In this case, an alternative is to assume, in addition, that $G$ is continuous on its domain as in \cite[Theorem 3.4]{malitsky2018first}.

\end{enumerate} 
\end{remark}

\subsection{A linesearch procedure for primal-dual splitting methods} \label{ss:LS-implementation}

In this section, we propose a natural implementation in Subroutine~\ref{LS:0} of \texttt{LS} to satisfy Assumptions~\ref{assumption-LS} and~\ref{assumption-LS-iii}. 

 \renewcommand{\algorithmcfname}{Subroutine}
\begin{algorithm}[!htb]
\caption{Linesearch procedure \texttt{LS} for Algorithm~\ref{a:P-D}.\label{LS:0}}

\textbf{Input}: \texttt{param} $=\{\beta >0, \delta_L \in (0,1), \mu \in (0,1)\}$, $\tau_{k(0)}$, $\tau_{k-1}$, $u^{k-1}$, $u^k$, $v^k$.

\For{$j= 0,1, \dots$}{
\step{Primal-dual trial points} Set $\theta_{k(j)} = \tau_{k(j)} \tau_{k-1}^{-1}$, and \begin{align*}
    \overline{u}^{k(j)}  &= u^k + \theta_{k(j)}(u^k - u^{k-1}) \\
    v^{k(j)+1} &= \prox_{\beta \tau_{k(j)} F}\big( v^k + \beta \tau_{k(j)} [K\overline{u}^{k(j)} - \nabla H(v^k)]\big)
\end{align*}

\step{Backtracking linesearch test} If \begin{equation} \label{not-LS-cond}
            \tau_{k(j)}\left( H(v^{k(j)+1}) - H(v^k) - \langle \nabla H(v^k), v^{k(j)+1} - v^k \rangle \right) > \dfrac{\deltaL}{2\beta} \|v^{k(j)+1} - v^k\|^2,
        \end{equation}  set $\tau_{k(j+1)} = \mu \tau_{k(j)}$, update $j \leftarrow j+1$, and go back to Step~1. Otherwise, set $\tau_{k} = \tau_{k(j)}$, $v^{k+1} = v^{k(j)+1} $, and exit the loop. 
}
\textbf{Output}: $\tau_k, v^{k+1}$. 
\end{algorithm}

\begin{remark}

Some comments regarding the parameters in Subroutine~\ref{LS:0} are in order. 
\begin{enumerate}
\item The shrinking parameter $\mu \in (0,1)$ is used in the backtracking linesearch to find a stepsize of adequate length  to satisfy \eqref{LS-term}. As for $\beta >0$, it is used to balance the primal and dual stepsizes. The parameter sequence $(\alpha_k)$ and the parameter $\gamma \in (0, 1)$  give flexibility to the stepsizes. The initialisation in \eqref{init-tau0}   assures that the stepsize sequence remains bounded, but possibly increase it if larger stepsizes are allowed. 

\item  Suppose that, in \eqref{init-tau0},  $\alpha_k = 1$ for all $k \geq 1$, and that the linesearch \texttt{LS} loop terminates in a finite number of steps in each iteration. Then, the sequence of stepsizes $(\tau_k)$  is nonincreasing and bounded above.  In this case, it suffices to initialise the outer loop with $\tau_{0} = \frac{\sqrt{\deltaB}}{\sqrt{\beta}\|K\|}$, and set $\tau_{k(0)} = \tau_{k-1}$ for condition \eqref{init-tau0} to always hold. Furthermore, the stepsize sequence is also eventually constant, and hence our proposed method eventually reduces to \cite[Algorithm 3.2]{condat2013primal}. To see this, suppose, to the contrary, that the stepsize sequence decreases an infinite number of times. By construction, each decrease must be at least by a factor of $\rho\in(0,1)$. This implies that $(\tau_k)$ converges to $0$, a contradiction.

\item Let us assume that after finitely many iterations of Algorithm~\ref{a:P-D}, the stepsizes $
\tau_k$ become constant equal to $\tau$. In view of \eqref{init-tau0}, it holds $ \beta \tau^2 \|K\|^2 \leq \deltaB$. Furthermore, the linesearch condition \eqref{not-LS-cond} suggest that an approximation of the equivalent of a Lipschitz constant is $\frac{\deltaL}{\beta\tau}$. Therefore, for $\sigma = \beta \tau$, the condition on the stepsize of the Condat--V\~u method reads \[ \frac{\frac{\deltaL}{\beta\tau} \sigma}{2} + \beta \tau^2 \|K\|^2 \leq \frac{\deltaL}{2} + \deltaB < 1 - \frac{\deltaL}{2}. \] Hence, if the linesearch condition eventually finds constant stepsizes, it implies the bounds on the stepsizes in~\cite{condat2013primal,vu2013splitting}, as the right-hand side in the above estimate is strictly smaller than $1$. Therefore, we retrieve convergence to saddle points for this special case.

\end{enumerate}
\end{remark}

If $\nabla H$ is globally Lipschitz continuous with constant $L>0$, the inequality \eqref{LS-term} holds whenever $\tau_{k} \leq \frac{\deltaL}{\beta L}$ by virtue of Lemma~\ref{descent-lemma}. This estimate is valid after finitely many steps in the linesearch procedure, since $(\tau_{k(j)})_{j\in \N}$ is a decreasing sequence and hence the linesearch is well-defined for globally Lipschitz continuous gradients.  The analogous result when $\nabla H$ is only locally Lipschitz continuous is presented in the following lemma.

\begin{lemma} \label{lemma-asmp-hold}
Suppose that $F: \mathcal{V} \to \R \cup \{+\infty\} $ and $G: \mathcal{U} \to \R \cup \{+\infty\}$ are proper lsc convex functions, and $H: \mathcal{V} \to \R$ is a differentiable convex function with locally Lipschitz continuous gradient. Then{, for the sequences $(u^k)$, $(v^k)$ and $(\tau_k)$  generated by Algorithm~\ref{a:P-D} using Subroutine~\ref{LS:0}}, the following hold:

\begin{enumerate}
    \item For each $k \geq 1$, there exists $\xi_k >0$ such that \begin{equation} \label{LS-term-t} 
\forall t \in (0, \xi_k],\quad
t \left( H(v^{k+1}(t)) - H(v^k) - \langle \nabla H(v^k), v^{k+1}(t) - v^k \rangle \right) \leq \dfrac{\deltaL}{2\beta} \|v^{k+1}(t) - v^k\|^2,
\end{equation} 
where $\overline{u}^{k}(t) = u^k + t\tau_{k-1}^{-1}(u^k - u^{k-1}) $ and $v^{k+1}(t) = \prox_{\beta t F}( v^k + \beta t [K\overline{u}^{k}(t) - \nabla H(v^k)])$.

 \item The linesearch in Subroutine~\ref{LS:0} terminates in finitely many iterations with $\tau_k \leq \tau_{k(0)}$.

\item The sequence $(\tau_k)$  
is bounded and separated from $0$. 
\end{enumerate}
   In particular, the implementation of \texttt{LS} in Subroutine~\ref{LS:0} is well-defined and satisfies Assumptions~\ref{assumption-LS} and~\ref{assumption-LS-iii}. 
\end{lemma}

\begin{proof}
    In order to prove (i), fix $k\geq 1$. In view of Lemma~\ref{lemma:prox-extension}, there exists $\eta_k >0$ such that, for all $t \in (0, \eta_k]$, we have $v^{k+1}(t) \in C_k :=\proj_{\overline{\dom(\partial F)}}(v^k) + B(0, 1)$, where  $B(0,1)$ denotes the open ball centred at $0$ with radius $1$.   Moreover, as $\nabla H$ is locally Lipschitz continuous and $\mathcal{V}$ is finite dimensional, $\nabla H$ is Lipschitz continuous on the open bounded convex set $C_k$. Thus, according to Lemma~\ref{descent-lemma}, there exists $L_k >0$ such that 
$$ \forall v\in C_k,\quad \left( H(v) - H(v^k) - \langle \nabla H(v^k),v - v^k \rangle \right) \leq \frac{L_k}{2} \|v - v^k\|^2. $$ Taking $\xi_k = \min\{\eta_k, \frac{\deltaL}{\beta L_k}\}$ yields \eqref{LS-term-t}.  Since $\mu \in (0,1)$, there exists $j_k \geq 1$ such that $\tau_{k(j)} = \tau_{k(0)} \mu^{j}\leq \xi_k$ for all $j \geq j_k$. This shows that the linesearch terminates after $j_k$ iterations. Furthermore, since $\tau_k = \tau_{k(j_k)} = \tau_{k(0)}\mu^{j_k} \leq \tau_{k(0)}$, then Assumption~\ref{assumption-LS}(i) holds with $\tau_{\text{init}} = \tau_{k(0)}$, proving (ii). By taking $t=\tau_k$ in \eqref{LS-term-t}, Assumption~\ref{assumption-LS}(ii) follows{, and thus Assumption~\ref{assumption-LS} holds}. Next, as a consequence of \eqref{init-tau0}, $(\tau_k)$ is bounded above by $\sqrt{\deltaB}/(\sqrt{\beta}\|K\|)$. Suppose, by way of a contradiction, that $(\tau_k)$ is not separated from zero. Then there exists a decreasing subsequence $(\tau_{k_\ell})_\ell$ with $\tau_{k_\ell}\to 0$ as $\ell\to+\infty$. Set $\hat{\tau}_{k_\ell}=\mu^{-1}\tau_{k_\ell}$ and set
$$ \hat{u}^{k_\ell}=u^k+\hat{\tau}_{k_\ell}\tau_{k-1}^{-1}(u^k-u^{k-1}),\quad \hat{v}^{k_\ell+1}=\prox_{\beta \hat{\tau}_{k_\ell}F}\big( v^k + \beta \hat{\tau}_{k_\ell} [K\hat{u}^{{k_\ell}} - \nabla H(v^{k_\ell})]\big).$$
Due to the linesearch procedure in Subroutine~\ref{LS:0}, we obtain
\begin{equation}\label{eq:hat tau}
\dfrac{\deltaL}{2\beta}\| \hat{v}^{{k_\ell}+1} - v^{k_\ell} \|^2 < \hat{\tau}_{k_\ell}\left( H(\hat{v}^{{k_\ell}+1}) - H(v^{k_\ell}) - \langle \nabla H(v^{k_\ell}), \hat{v}^{{k_\ell}+1} - v^{k_\ell} \rangle \right). 
\end{equation}
{As Assumption~\ref{assumption-LS} holds, b}y Lemma~\ref{prop:bounded}, the sequences $(u^k)$ and $(v^k)$ are bounded. Thus, appealing to Lemma~\ref{lemma:prox-extension}, $(\hat{v}^{k_\ell+1})_\ell$ is bounded as the continuous image of a bounded sequence. Let $C\subseteq\mathcal{V}$ denote any bounded open convex set containing $(v^{k_\ell})_\ell\cup(\hat{v}^{k_\ell+1})_l$. By applying Lemma~\ref{descent-lemma}, there exists $L_C>0$ such that
\begin{equation}\label{eq:hat tau des}
    H(\hat{v}^{k_\ell+1}) - H(v^{k_\ell}) - \langle \nabla H(v^{k_\ell}), \hat{v}^{k_\ell+1} - v^{k_\ell} \rangle \leq \dfrac{L_C}{2}\| \hat{v}^{k_\ell+1} - v^{k_\ell} \|^2.
\end{equation}
Combining \eqref{eq:hat tau} and \eqref{eq:hat tau des} yields the contradiction $\hat{\tau}_{k_\ell}>\deltaL/(\beta L_C)$, and hence we conclude that $(\tau_k)$ is separated from zero. The fact that $\theta_k=\tau_k/\tau_{k-1}$ is also bounded above and separate from zero follows, showing (iii) holds, from where Assumption~\ref{assumption-LS-iii} follows.
\end{proof}

\begin{corollary}[Primal-dual convergence] 
\label{th:convergence-one-agent}
Suppose that $F: \mathcal{V} \to \R \cup \{+\infty\}$ and $G: \mathcal{U} \to \R \cup \{+\infty\}$ are proper lsc convex functions, and $H: \mathcal{V} \to \R$ is a differentiable convex function with locally Lipschitz continuous gradient. Then, the sequence $(u^k, v^k)$ generated by Algorithm~\ref{a:P-D} with \texttt{LS} implemented with Subroutine~\ref{LS:0}, converges to saddle point $(\hat{u}, \hat{v})$ of problem \eqref{MP-problem}, and  $\left(\hat{\mathcal{G}}_{(\hat{u}, \hat{v})}(u^k,v^{k})\right)$ converges to $0$. 
\end{corollary}

\begin{proof}
    In view of  Lemma~\ref{lemma-asmp-hold}, Assumption~\ref{assumption-LS}  holds, and thus from Lemma~\ref{prop:bounded}(ii), the generated sequences are bounded. Furthermore, from Lemma~\ref{lemma-asmp-hold}, Assumption~\ref{assumption-LS-iii} holds as well.  Therefore, the  result follows from Theorem~\ref{th:abstract-convergence}.
\end{proof}

\begin{remark} \label{alternative-LS}

In the following, we comment on the above convergence result.

\begin{enumerate} 

\item {Algorithm~\ref{a:P-D} implemented with Subroutine~\ref{LS:0} corresponds to \cite[Algorithm 4]{malitsky2018first} when the norm of $K$ is known, as in  Section~\ref{s:distributed}. This case was not originally covered in \cite{malitsky2018first}, and we obtain similar convergence results (cf. \cite[Theorems~3.4 \&~3.5]{malitsky2018first}) under a weaker assumption on the gradient of $H$. }

\item
If the operator norm of $K$ is not available for use in \eqref{init-tau0}, then it is instead possible to replace \eqref{init-tau0} with $\tau_{k(0)}=\tau_{k-1}\alpha_k$ and modify the linesearch conditions in \eqref{not-LS-cond} with\begin{multline} \label{not-LS-cond-2}
             \frac{\tau_{k(j)}^2}{2} \|K^* v^{k(j)+1} - K^* v^k\|^2 + \tau_{k(j)}\left( H(v^{k(j)+1}) - H(v^k) - \langle \nabla H(v^k), v^{k(j)+1} - v^k \rangle \right)\\ > \dfrac{ \deltaB+\deltaL}{2\beta} \|v^{k(j)+1} - v^k\|^2.
\end{multline}This corresponds to the linesearch proposed and analysed in \cite[Algorithm 4]{malitsky2018first} for globally Lipschitz $\nabla H$. In the locally Lipschitz case, it is also possible to prove analogous results to Lemma~\ref{lemma-asmp-hold}(i) and (ii), and~\ref{prop:bounded}. Regarding Lemma~\ref{lemma-asmp-hold}(iii), since we modify the initialisation in \eqref{not-LS-cond}, we can only guarantee that $(\tau_k)$ is separated from zero, which is enough to argue  convergence to saddle points as in \cite[Section 4]{malitsky2018first}.

\item  In the presence of additional structure in $H$, other variants of Subroutine~\ref{LS:0} can be used. This will be discussed further in Section~\ref{s:distributed}.

\end{enumerate}
\end{remark}

In this section, we defined a concrete linesearch procedure for Algorithm~\ref{a:P-D}  that leads to convergence to saddle points of the min-max convex-concave problem \eqref{MP-problem}, by assuming that the norm of the operator $\|K\|$ is available. In the next section, we apply these ideas to obtain implementable distributed algorithms for multi-agent minimisation problems.


\section{A distributed proximal-gradient algorithm with linesearch}
\label{s:distributed}
In this section, we propose a distributed first-order method with linesearch based upon Algorithm~\ref{a:P-D} to solve the multi-agent minimisation problem \eqref{primal-problem}. In doing so, the main question that arises is how to implement the linesearch in distributed settings. We propose two implementations  that address this with differing amounts of communication.

\subsection{Construction of the algorithm}
Let $W\in\mathbb{R}^{n\times n}$ be a mixing matrix, and set $U=\big(\frac{I-W}{2}\bigr)^{1/2}$. As explained in Section~\ref{s:prelim}, problem~\eqref{primal-problem} can be formulated as 
\begin{equation} \label{primal-dual-linear-coupling-problem} 
\displaystyle\min_{\bx \in \R^{n \times d}}\max_{\by \in \R^{n \times d}}  \langle U\by,\bx \rangle + f(\bx) + h(\bx).
\end{equation} 
Hence, problem~\eqref{primal-dual-linear-coupling-problem} is in the form of problem~\eqref{MP-problem} of Section~\ref{s:splitting-LS} with
\begin{equation} \label{eq:identifications}
    u = \by, \:  v = \bx, \: K = -U,  \: G = \iota_{0}^{*}  \equiv 0, \: F = f \text{ and } H = h.
\end{equation}
Applying Algorithm~\ref{a:P-D} to problem \eqref{primal-dual-linear-coupling-problem} gives 
\begin{equation} \label{PGEXTRA-like}
    \left\{\begin{array}{rcl}
        \by^k & = &  \by^{k-1} + \tau_{k-1} U\bx^k \\
        \overline{\by}^k & = &  \by^k + \theta_{k} (\by^k - \by^{k-1}) \\
        \bx^{k+1} & = &  \prox_{\beta\tau_k f}\Big( \bx^k - \beta\tau_k \big(U \overline{\by}^k + \nabla h(\bx^k) \big) \Big),
    \end{array}\right.
\end{equation} 
where $\theta_k = \tau_k \tau_{k-1}^{-1}$ and $\tau_k$ is calculated using a linesearch. However, in this form, \eqref{PGEXTRA-like} is not suitable for a distributed implementation since $U$ need not be a mixing matrix (in particular, it need not satisfy Definition~\ref{def:mixing-matrix}(i)). To address this, denote
\begin{equation} \label{change-of-variables}
     \bu^k= U\by^k \text{ and } \overline{\bu}^k= U\overline{\by}^k.
 \end{equation}
 Premultiplying the first two equations in \eqref{PGEXTRA-like} by $U$ and substituting \eqref{change-of-variables} yield
 \begin{equation} \label{PGEXTRA-like-2}
    \left\{\begin{aligned}
         \bu^k& = \bu^{k-1}+ \frac{\tau_{k-1}}{2}(I - W)\bx^k  \\
         \overline{\bu}^k &= \bu^k + \theta_k(\bu^k - \bu^{k-1})\\
         \bx^{k+1} & = \prox_{\beta\tau_k  f}\big(\bx^k- \beta\tau_k[\overline{\bu}^k + \nabla h(\bx^k)]\big),
    \end{aligned}\right.
\end{equation}
Since this only involves the mixing matrix $W$, \eqref{PGEXTRA-like-2} is suitable for distributed implementation. Due to \eqref{change-of-variables}, we also note that the initialisation for \eqref{PGEXTRA-like-2} requires that $\bu^0\in\range U$ (\emph{e.g.,} $\bu^0=0$).

As a result of the discussion above, we present, in Algorithm~\ref{a:PGE}, a distributed proximal-gradient method for finding solutions to problem \eqref{primal-dual-linear-coupling-problem}, given an abstract distributed linesearch procedure \texttt{LS} that computes stepsizes $\tau_k$ satisfying
\begin{equation*} 
\tau_{k}\left( h(\bx^{k+1}) - h(\bx^k) - \langle \nabla h(\bx^k), \bx^{k+1} - \bx^k \rangle \right) \leq \dfrac{\deltaL}{2\beta} \|\bx^{k+1} - \bx^k\|^2.
\end{equation*}
This method corresponds to Algorithm~\ref{a:P-D} applied to \eqref{primal-dual-linear-coupling-problem}. In Section~\ref{s:LS-distributed}, we discuss two possible distributed implementations of \texttt{LS}. 

 \renewcommand{\algorithmcfname}{Algorithm}

\begin{algorithm}[!htb]
\caption{Distributed proximal-gradient algorithm for problem~\eqref{primal-problem}.\label{a:PGE}}
Choose a mixing matrix $W\in\mathbbm{R}^{n\times n}$\;
Choose a set of parameters \texttt{param} containing $\beta >0$ and $ \deltaL \in (0,1)$\;
Choose parameters $\tau_0 >0$, $\gamma \in (0,1),$ and $\deltaB \in (0,1)$ such that $\deltaB + \deltaL < 1$\;
Set $\theta_0=1$\;
Choose an initial point $\bx^1\in\R^{n \times d}$ and set $\bu^0 = 0$\;
\For{$k=1,2,\dots,$}{
\step{Dual update} Update variables according to
\begin{equation*} 
    \bu^k = \bu^{k-1} + \frac{1}{2}\tau_{k-1}(I-W) \bx^k.
\end{equation*}

\step{Backtracking linesearch and primal update} Initialise the linesearch by choosing  $\alpha_{k} \in[1,  \sqrt{1 + \gamma\theta_{k-1}} ]$, and define
\begin{equation} \label{init-LS}
        \tau_{k(0)} = \min\left\{\frac{\sqrt{2\deltaB}}{\sqrt{\beta(1 - \lambda_{\min}(W))}},\tau_{k-1}\alpha_{k}\right\}.
    \end{equation}  

    Compute \begin{equation*}
        \tau_k, \bx^{k+1} = \texttt{LS}(\texttt{param}, \tau_{k(0)}, \tau_{k-1},\bu^{k-1},\bu^k, \bx^k),
    \end{equation*} 
    
    where
\begin{align*}
    \theta_k &= \tau_{k}\tau_{k-1}^{-1} \\
    \overline{\bu}^{k}  &= \bu^k + \theta_k(\bu^k - \bu^{k-1}) \\
       \bx^{k+1} &= \prox_{\beta \tau_{k} f}\big( \bx^k - \beta \tau_{k} [\overline{\bu}^{k} + \nabla h(\bx^k)]\big).
    \end{align*} 
      }      
\end{algorithm}

Note that the linesearch initialisation~\eqref{init-tau0} for Algorithm~\ref{a:P-D} involves the opearator norm of $K = -U$. In view of the definition of $U$, this can be expressed in terms of the minimum eigenvalue of the mixing matrix $W$:
\begin{equation*} 
    \| K \|^2 = \| U^2 \| =  \dfrac{1}{2}\| I - W \| = \dfrac{1}{2}(1 - \lambda_{\min}(W)).
\end{equation*}
This identity is used to obtain the initialisation rule for $\tau_{k(0)}$ in \eqref{init-LS} from \eqref{init-tau0}.

\begin{remark}
Some further comments regarding Algorithm~\ref{a:PGE} are in order. Although the initialisation $\bu^0=0$ is used in Algorithm~\ref{a:PGE} for convenience,  it is actually possible to use $\bu^0=(u_1^0,\dots,u_n^0)^\top$ such that $\sum_{i=1}^nu_i^0=0$. To see this, note that the derivation~\eqref{change-of-variables} only requires $ \bu^0\in \range U$ and, due to Definition~\ref{def:mixing-matrix}(iii), we have
$$ \range U = \range U^2 = \range (I-W) = \ker(I-W)^\perp =\{\bu=(u_1,\dots,u_n)^\top:\sum_{i=1}^nu_i=0\}. $$
The linesearch procedure \texttt{LS} in Algorithm~\ref{a:PGE} is assumed to perform agent-independent proximal steps to define
\begin{equation*}
        \bx^{k+1} = \prox_{\beta \tau_k f}(\bx^k - \beta \tau_k [\bar{\bu}^k + \nabla h(\bx^k)]).
\end{equation*}
Two possible ways to implement \texttt{LS} are described in Subroutines~\ref{LS:1} and~\ref{LS:2} assuming $\nabla h_i$ are locally Lipschitz continuous for each $i \in \{1, \dots, n\}$ and that (an estimate of) $\lambda_{\min}(W)$ is available.

At the cost of extra communication, a variation of the proposed linesearch based on \cite[Section 5]{malitsky2018first} is discussed in Remark~\ref{remark:comm-in-LS}, which does not require knowledge of $\lambda_{\min}(W)$ in the initialisation~\eqref{init-tau0}.
\end{remark}

\begin{remark}[Equivalent forms of PG-EXTRA]\label{remark:PG-E}
In the following, we explain the relation between PG-EXTRA \eqref{PG-EXTRA} {from \cite{shi2015proximal}} and our proposal  \eqref{PGEXTRA-like-2} (Algorithm~\ref{a:PGE}).

\begin{enumerate}
\item 
Although not immediately obvious, when the stepsize is fixed, \eqref{PGEXTRA-like-2} is equivalent to PG-EXTRA~\eqref{PG-EXTRA} as we explain below. When the stepsize is not fixed, it is convenient to work with \eqref{PGEXTRA-like-2} due to the role of the sequence $(\theta_k)$. 

To see the correspondence between \eqref{PGEXTRA-like-2} with fixed stepsize and \eqref{PG-EXTRA}, first let $\tau_k = \tau$ for some constant $\tau >0$ and all $k \geq 1$. Then $\theta_k = 1$. Next, define $\beta = \tau^{-2}$, and denote the argument of the proximal step in \eqref{PGEXTRA-like-2} by 
$$\bw^k = \bx^k- \beta\tau[\overline{\bu}^k + \nabla h(\bx^k)] \quad\forall k \geq 1.$$
Noting that $\beta\tau=\tau^{-1}$, for all $k\geq 2$, we have
\begin{equation*}\begin{aligned}    
\bw^{k} - \bw^{k-1}     
&=   \bx^{k}-\bx^{k-1} - \beta\tau \bigl(\bar{\bu}^k-\bar{\bu}^{k-1}\bigr) - \beta\tau\bigl(\nabla h(\bx^k)-\nabla h(\bx^{k-1})\bigr) \\     
&=  \bx^{k}-\bx^{k-1} - \tau^{-1}\bigl(2(\bu^k-\bu^{k-1}) - (\bu^{k-1} - \bu^{k-2})\bigr) - \beta\tau\bigl(\nabla h(\bx^k)-\nabla h(\bx^{k-1})\bigr) \\ 
&=  \bx^{k}-\bx^{k-1} - \tau^{-1} \left(\tau(I-W)\bx^k-\frac{\tau}{2}(I-W)\bx^{k-1}\right) - \beta\tau\bigl(\nabla h(\bx^k)-\nabla h(\bx^{k-1})\bigr) \\     
&=  W\bx^k - \dfrac{1}{2}(I+W)\bx^{k-1} - \beta\tau\bigl(\nabla h(\bx^k)-\nabla h(\bx^{k-1})\bigr).
\end{aligned}
\end{equation*}
Altogether, we have
\begin{equation*} 
    \left\{\begin{aligned}
         \bw^k& = \bw^{k-1} + W\bx^k   - \dfrac{1}{2}(W+I)\bx^{k-1} - \sigma\left( \nabla h(\bx^k)-\nabla h(\bx^{k-1})\right)  \\
         \bx^{k+1} & = \prox_{\sigma  f}\big(\bw^k\big),
    \end{aligned}\right.
\end{equation*}
which corresponds to \eqref{PGEXTRA-like-2} with stepsize $\sigma = \beta\tau=\tau^{-1}$.  

\item 
The initialisation of Algorithm~\ref{a:PGE} can recover the initialisation of PG-EXTRA used in \cite{shi2015proximal}. Indeed, taking $\bu^0=0$ and any $\bx^1 \in \R^{n \times d}$ gives $\bu^1=\frac{\tau}{2}(I-W)\bx^1$ which implies
$ \overline{\bu}^1 = \bu^1+(\bu^1-\bu^0) = \tau(I-W)\bx^0$. 
Since $\beta\tau^2=1$, we therefore have
 $$ \bx^2 = \prox_{\beta\tau  f}\big(\bx^1- \beta\tau[\overline{\bu}^1 + \nabla h(\bx^1)]\big) = \prox_{\beta\tau  f}\big(W\bx^1- \beta\tau\nabla h(\bx^1)\big). $$ Alternatively, taking any $\bx^1 \in \R^{n \times d}$ and setting $\bu^0 = - \tau(I-W)\bx^1$ gives $\bu^1 = - \frac{\tau}{2}(I-W)\bx^1$ which implies $\overline{\bu}^1 = 0$. We therefore obtain 
 $$\bx^2 = \prox_{\beta\tau  f}\big(\bx^1- \beta\tau\nabla h(\bx^1)\big),$$
 which corresponds to the initialisation of PG-EXTRA discussed in \cite[Remark~4.5]{malitsky2023first}. 
\end{enumerate}
\end{remark}

\subsection{Linesearch procedures in distributed settings} \label{s:LS-distributed}
In this section, we propose two distributed linesearch techniques that can take the role of \texttt{LS} in Algorithm~\ref{a:PGE}. Although the updates for the variables $\bx^k,\bu^k$ and $\bar{\bu}^k$ in Algorithm~\ref{a:PGE} can be decentralised (when the stepsize is known), the proposed linesearch procedure internally requires global communication of scalar values across all agents. More precisely, \texttt{LS} as defined in Subroutine~\ref{LS:1} requires a method for computing a global (scalar) sum, and \texttt{LS} as defined in Subroutine~\ref{LS:2} requires computing a minimum (scalar) value among agents. Both of these operators can be efficiently implemented, for instance, using \emph{Message Passing Interface (MPI)} \cite{mpi41}.  See \cite[Chapter 10]{boyd2011distributed} for more details. Since the aforementioned approaches only require communication of scalars across the network, their global communication costs are small compared to methods that need global communication of the vectors and matrix variables such as $\bx^k,\bu^k$ and $\bar{\bu}^k$. 

\paragraph{Linesearch with (scalar) summation {over agents}.} 
The first distributed linesearch subroutine  we propose is presented in Subroutine~\ref{LS:1}, and corresponds to a direct application of Subroutine~\ref{LS:0}.  It locally computes independent proximal steps for each agent in \eqref{prox:k(j)}, using one common stepsize. In each outer loop of the linesearch, a round of global communication is then used to compute the sum of the scalars $a_1,\dots,a_n$, and the linesearch condition \eqref{not-LS-i} is checked. Note that this check (for positivity of the sum) can either be performed by all agents after the global value has been returned, or performed by one agent followed by broadcasting the outcome to the other agents. 
Observe also that each computation of the summation~\eqref{not-LS-i} involves performing $n$ proximal steps (one per agent). 
Note that, when \eqref{not-LS-i} holds, the current stepsize is too large and must be shrunk by the same amount for all agents before returning to Step~1. When \eqref{not-LS-i} does not hold, the linesearch procedure terminates with all agents agreeing on a common stepsize. 

  \renewcommand{\algorithmcfname}{Subroutine}
  
\begin{algorithm}[!htb]
\caption{Linesearch \texttt{LS} with sum {over agents} for Algorithm~\ref{a:PGE}.\label{LS:1}}

\textbf{Input}: \texttt{param} $=\{\beta >0, \delta_L \in (0,1), \mu \in (0,1)\}$, $\tau_{k(0)}$, $\tau_{k-1}$, $\bu^{k-1}$, $\bu^k$, $\bx^k$.

\For{$j= 0,1, \dots$}{
\step{Primal-dual agent-wise operations}

\For{$i= 1, \dots, n$}{
 Define $\overline{u}^{k(j)}_i = u^k_i + \tau_{k(j)}\tau_{k-1}^{-1}(u^k_i - u^{k-1}_i)$. Compute \begin{equation} \label{prox:k(j)}
        x^{k(j)+1}_i = \prox_{\beta \tau_{k(j)} f_i}\big( x^k_i - \beta \tau_{k(j)}[\overline{u}^{k(j)}_i + \nabla h_i(x^k_i)]\big),
    \end{equation} and \begin{equation*}
        a_i =  \tau_{k(j)}[h_i(x_i^{k(j)+1}) - h_i(x_i^k) - \langle \nabla h_i(x_i^k), x_i^{k(j)+1} - x_i^k \rangle ] - \dfrac{\deltaL}{2\beta}\displaystyle\| x_i^{k(j)+1} - x_i^k \|^2.
    \end{equation*}
}  
\step{Check the linesearch condition} Compute the sum $\sum_{i=1}^n a_i$, and  if \begin{equation}\label{not-LS-i}
        \begin{array}{c} \displaystyle\sum_{i=1}^n a_i
             >  0,
        \end{array}
    \end{equation}
    set $\tau_{k(j+1)} =\mu\tau_{k(j)}$ and go back to Step 1. Otherwise, set $\tau_k = \tau_{k(j)}$,  $\bx^{k+1} =  \bx^{k(j)+1}$, and exit the linesearch. 
}
\textbf{Output}: $\tau_k, \bx^{k+1}$. 

\end{algorithm}

We also propose a linesearch that instead of requiring the calculation of a sum of local values over a network, it needs the computation of the minimum (scalar) value among agents in a network.

\paragraph{Linesearch with a global minimum.} 

The second distributed linesearch subroutine we propose is presented in Subroutine~\ref{LS:2}. Different from Subroutine~\ref{LS:1}, this is not a direct application of Subroutine~\ref{LS:0} as it locally computes independent proximal steps for each agent in \eqref{prox:k(j)-i} using potentially different stepsizes between agents. In the outer loop (Steps~1~\&~2), the linesearch condition \eqref{LS-i-b} is tested locally for each agent, rather than globally as in Subroutine~\ref{LS:1}, and hence no global communication is need in this part. In exchange, a common stepsize for all agents is computed  by finding a global minimum across the network. This common stepsize is then used to recompute the proximal steps for each agent (if needed). Note that  the global minimum only needs to be computed once in each time \texttt{LS} is run.

\begin{algorithm}[!htb]
\caption{Linesearch \texttt{LS} with a  global minimum for Algorithm~\ref{a:PGE}.\label{LS:2}}

\textbf{Input}: \texttt{param} $=\{\beta >0, \delta_L \in (0,1), \mu \in (0,1)\}$, $\tau_{k(0)}$, $\tau_{k-1}$, $\bu^{k-1}$, $\bu^k$, $\bx^k$.

\step{Primal-dual and linesearch agent-wise operations}

\For{$i= 1, \dots,n$}{

Set $\tau_{k(0),i} = \tau_{k(0)}$.

\For{$j= 0,1, \dots$}{

Step 1.1. Define $\check{u}^{k(j)}_i = u^k_i + \tau_{k(j),i}\tau_{k-1}^{-1}(u^k_i - u^{k-1}_i)$. Compute
\begin{equation} \label{prox:k(j)-i}
        \check{x}^{k(j)+1}_i = \prox_{\beta \tau_{k(j),i} f_i}\big( x^k_i - \beta \tau_{k(j),i}[\overline{u}^{k(j)}_i + \nabla h_i(x^k_i)]\big).
    \end{equation} and \begin{equation*}
        b_i = \tau_{k(j),i}\left[h_i(\check{x}_i^{k(j)+1}) - h_i(x_i^k) - \langle \nabla h_i(x_i^k), \check{x}_i^{k(j)+1} - x_i^k \rangle \right] 
             - \dfrac{\deltaL}{2\beta}\| \check{x}_i^{k(j)+1} - x_i^k \|^2.
    \end{equation*}
    
Step 1.2. Backtracking linesearch test. If \begin{equation}\label{LS-i-b}
        \begin{array}{c} b_i 
             > 0,
        \end{array}
    \end{equation} 
    set $\tau_{k(j+1),i} =\mu\tau_{k(j),i}$ and go back to Step 1.1. Otherwise, set $\tau_{k,i}= \tau_{k(j),i}$, ${\overline{u}}_i^k = \check{u}_i^{k(j)}$, $x_i^{k+1} = \check{x}_i^{k(j)+1}$, and finish the linesearch for agent $i$. 
    }
}
\step{} Compute $\tau_k = \displaystyle\min_{i=1,\dots,n} \tau_{k,i}$ and 

\For{$i = 1, \dots, n$ such that $\tau_{i,k} > \tau_k$}{
Recompute 
\begin{equation*}
    \begin{aligned}
        \overline{u}^{k}_i & =  u^k_i + \tau_{k}\tau_{k-1}^{-1}(u^k_i - u^{k-1}_i)  \\
        x^{k+1}_i &= \prox_{\beta \tau_{k} f_i}\big( x^k_i - \beta \tau_{k}[\overline{u}^{k}_i + \nabla h_i(x^k_i)]\big).
    \end{aligned}
\end{equation*}
}
\textbf{Output}: $ \tau_k, \bx^{k+1}$. 
\end{algorithm}

\begin{remark} \label{remark:comm-in-LS}
Some comments regarding Subroutines~\ref{LS:1} and~\ref{LS:2} are in order.

\begin{enumerate}
\item
The rounds of communication in both algorithms are different in nature. On the one hand, in each outer loop of Subroutine~\ref{LS:1}, we require communication to compute the global sum in \eqref{not-LS-i} and test the inequality therein. On the other hand, in Subroutine~\ref{LS:2}, we require communication to compute the minimum stepsize among agents only once per linesearch call. 

\item 
One advantage of Subroutine~\ref{LS:1} is that it does not require the computation of additional proximal steps after communicating, as all agents share the same stepsize before entering a new inner linesearch iteration. However, this cost is balanced by the global communication needed to check the linesearch condition. In this regard, there is a trade-off between checking a global sum and recomputing proximal steps and computing the global minimum stepsize. The latter is the case of Subroutine~\ref{LS:2}, in which both proximal steps and test of the linesearch condition are performed locally, and may be more suitable whenever the proximal steps are cheap relatively to the cost of computing global scalar sums.  

{\item Step~1 of Subroutine~\ref{LS:1} and Subroutine~\ref{LS:2} have the following operations in common before checking the corresponding linesearch test: $n$ function evaluations, $n$ proximal-gradient evaluations and $2n$ inner products of vectors of dimension $d$. Observe that Subroutine~\ref{LS:2} also performs extra $O(n)$ proximal-gradient evaluations in Step 2, in exchange of globally checking the linesearch test as in Subroutine~\ref{LS:1}.}

\end{enumerate}

\end{remark}

\subsection{Convergence analysis}

In this section, we analyse convergence of Algorithm~\ref{a:PGE} with \texttt{LS} being either Subroutine~\ref{LS:1} or~\ref{LS:2}, which will require the following lemma. 

\begin{lemma} \label{lemma:distrib-LS-terminates}
    Suppose that $f_1,\dots,f_n: \R^{d} \to \R \cup \{+\infty\}$ are proper lsc convex functions and $h_1,\dots,h_n: \R^{d} \to \R $ are convex differentiable functions with locally Lipschitz continuous gradients. Then, Subroutines~\ref{LS:1} and~\ref{LS:2} are well-defined and satisfy Assumptions~\ref{assumption-LS} and~\ref{assumption-LS-iii}. In particular, \begin{equation} \label{LS-i}
        \forall k \in \mathbb{N}, \quad \tau_k \sum_{i=1}^n\left[h_i(x_i^{k+1}) - h_i(x_i^k) - \langle \nabla h_i(x_i^k), x_i^{k+1} - x_i^k \rangle \right] \leq \dfrac{\deltaL}{2\beta}\sum_{i=1}^n\| x_i^{k+1} - x_i^k \|^2.
    \end{equation}
\end{lemma}

\begin{proof}
    Subroutine~\ref{LS:1} is a distributed implementation of Subroutine~\ref{LS:0}, therefore the result for this procedure follows from Lemma~\ref{lemma-asmp-hold}, and \eqref{LS-i} follows from \eqref{LS-term}. Regarding Subroutine~\ref{LS:2}, using Lemma~\ref{lemma-asmp-hold}(i) with $H=h_i$, for each $i \in \{1, \dots, n\}$ and $k \geq 1$, there exists $\xi_{k,i} >0$ such that
    \begin{equation} \label{LS-i-t}
        \forall t \in (0,  \xi_{k,i}],\quad t \left(h_i(x_i^{k+1}(t)) - h_i(x_i^k) - \langle \nabla h_i(x_i^k), x_i^{k+1}(t) - x_i^k \rangle \right) \leq \dfrac{\deltaL}{2\beta}\| x_i^{k+1}(t) - x_i^k \|^2,
    \end{equation}
    where $x_i^{k+1}(t) = \prox_{\beta t f_i}(x_i^k - \beta t[\overline{u}_i^k + \nabla h_i(x_i^k)])$. After finitely many backtracking linesearch steps, the  search for agent $i$ terminates with $\tau_{k,i}\in (0,  \xi_{k,i}]$. Since $\tau_k=\min_{i=1,\dots,n}\tau_{k,i}$, this means that the linesearch terminates with $\tau_k \leq \tau_{k(0)}$, so Assumption~\ref{assumption-LS}(i) holds with $\tau_{\text{init}} = \tau_{k(0)}$.  By aggregating \eqref{LS-i-t} over all $i \in \{1 , \dots, n\}$ and taking $ t = \tau_k $, yields Assumption~\ref{assumption-LS}(ii) and \eqref{LS-i}. Furthermore,  in view of Lemma~\ref{prop:bounded}, $(\bx^k)$ is bounded. In order to check that $(\tau_k)$ is separated from  $0$, assume by way of a contradiction that there exists a subsequence $\tau_{k_{\ell}} \to 0$ as $\ell \to +\infty$. From Step~2 of Subroutine~\ref{LS:2}, there exists $i \in \{1, \dots, n\}$ such that $\tau_{k_{\ell},i} \to 0$.  Applying Lemma~\ref{lemma-asmp-hold}(iii) with $\tau_k = \tau_{k_{\ell},i}$ yields a contradiction, a Step~1 in Subroutine~\ref{LS:2} corresponds to Subroutine~\ref{LS:0} applied to agent $i$. Hence, Assumption~\ref{assumption-LS-iii} holds.
\end{proof}

Since we have established that Algorithm~\ref{a:P-D} is well-defined and both Subroutines~\ref{LS:1} and~\ref{LS:2} satisfy Assumptions~\ref{assumption-LS} and~\ref{assumption-LS-iii}, we are ready to state the convergence result: the primal sequence $(\bx^k)$ converges to a solution in consensus, being asymptotically feasible.

\begin{corollary}[Convergence to primal solutions] \label{c:convergence}
Suppose that, for each $i \in \{1, \dots, n\}$, $f_i: \R^d \to \R \cup \{+\infty\}$ is a proper lsc convex function and $h_i: \R^d \to \R$ is a differentiable convex with locally Lipschitz continuous gradient, and that the set of solutions of \eqref{primal-problem} is nonempty. Then, the  sequence $(\bx^k)$ generated by Algorithm~\ref{a:PGE} using the linesearch in Subroutine~\ref{LS:1} or~\ref{LS:2}, converges to a point in consensus $\hat{\bx}$, such that any row of $\hat{\bx}$ defines a solution to  \eqref{primal-problem}. 
\end{corollary}

\begin{proof}
In view of Lemma~\ref{lemma:distrib-LS-terminates}, the result follows from Theorem~\ref{th:abstract-convergence}:  $(\bx^k)$  converges to a solution to the primal problem associated to \eqref{primal-dual-linear-coupling-problem}, which corresponds to \eqref{primal-problem-2}.  
\end{proof}

\begin{remark} {In the following, we comment on the convergence result above and possible variations of Subroutines~\ref{LS:1} and \ref{LS:2}.}

\begin{enumerate}

\item {If the global Lipschitz constant of some gradients $ \nabla h_i$
are available, Subroutine~\ref{LS:2} can be modified as follows. Let $I_{\text{Lip}}$ be the index set for which the global Lipschitz constant $L_i$ of $\nabla h_i$ is known. We first modify \eqref{init-LS} by setting \begin{equation*}
    \tau_{k(0)}^{\prime} = \min\left\{\frac{\sqrt{2\delta_K}}{\sqrt{\beta(1 - \lambda_{\min}(W))}},\tau_{k-1}\alpha_{k}, \tau_L\right\},
\end{equation*}where $\tau_L = \min_{i \in I_{\text{Lip}}} \frac{\delta_L}{\beta L_i} $. Then, we perform Step 1 of Subroutine~\ref{LS:2} for agents $i\notin I_{\text{Lip}}$ only, since for $i\in I_{\text{Lip}}$ and any stepsize  $t \leq \frac{\delta_L}{\beta L_i}$, it follows from Lemma~\ref{descent-lemma} that \begin{equation*}
    t[h_i(x_i^{k(j)+1}) - h_i(x_i^k) - \langle \nabla h_i(x_i^k), x_i^{k(j)+1} - x_i^k \rangle ] - \dfrac{\delta_L}{2\beta}\displaystyle\| x_i^{k(j)+1} - x_i^k \|^2 \leq 0.
\end{equation*} Hence, $b_i \leq 0$ for $i \in I_{\text{Lip}}$. In this case, Lemma~\ref{lemma:distrib-LS-terminates} follows directly, and thus Corollary~\ref{c:convergence} also holds. Observe that, from construction, 
it is not straightforward to modify  Subroutine~\ref{LS:1} to leverage the additional information provided by $I_{\text{Lip}} \neq \emptyset$ in order to reduce the cost of the linesearch subroutine, as all agents use a common stepsize to perform Steps 1 and 2. When the original condition is satisfied, namely, $\sum_{i \notin I_{\text{Lip}}} a_i + \sum_{i \in I_{\text{Lip}}} a_i \leq  0$, it may allow  $\sum_{i \notin I_{\text{Lip}}} a_i >0$. Hence, testing $\sum_{i \notin I_{\text{Lip}}} a_i > 0$ in place of \eqref{not-LS-i}, results in a stricter linesearch test with no reduction of computational costs.
}

\item 
The initialisation of $\tau_{k(0)}$ in \eqref{init-LS} requires the knowledge of the smallest eigenvalue of $W$.  Instead, we could also perform the linesearch in \eqref{not-LS-cond-2} and initialise $\tau_{k(0)} = \tau_{k-1}\alpha_{k}$. For this alternative linesearch, we need a distributed way to calculate the leftmost term on the left-hand side  to make it implementable. More specifically, since $K = -U$, the extra term in the linesearch can be rewritten as 
\begin{equation} \label{LS-with-W}
    \begin{array}{rcl}
        \|K^* \bx^{k(j)+1} - K^* \bx^k\|^2 & = & \langle \bx^{k(j)+1} -  \bx^k, K^2(\bx^{k(j)+1} -  \bx^k) \rangle \\
         & =&  \dfrac{1}{2}\langle \bx^{k(j)+1} -  \bx^k,  (I-W)(\bx^{k(j)+1} -  \bx^k) \rangle\\
         & =&  \dfrac{1}{2}\left( \| \bx^{k(j)+1} -  \bx^k\|^2 - \langle W(\bx^{k(j)+1} -  \bx^k),  \bx^{k(j)+1} -  \bx^k  \rangle \right).
    \end{array}
\end{equation} Therefore, in each step of the linesearch we need to perform pre-multiplications by $W$, which amounts to rounds of communication. In this way, we check \eqref{not-LS-i} with
\begin{multline*}
    a_i = \dfrac{\tau_{k(j)}^2}{4} \left( \| x_i^{k(j)+1} -  x_i^k\|^2 - \left\langle \sum_{j \in N(i)} W_{ij}(x_j^{k(j)+1} - x_j^k),  x_i^{k(j)+1} -  x_i^k  \right\rangle \right) \\ +\tau_{k(j)}\left[h_i(x_i^{k(j)+1}) - h_i(x_i^k) - \langle \nabla h_i(x_i^k), x_i^{k(j)+1} - x_i^k \rangle \right] 
              -  \dfrac{\deltaB+\deltaL}{2\beta}\| x_i^{k(j)+1} - x_i^k \|^2,
\end{multline*} while \eqref{LS-i-b} requires \begin{multline*}
    b_i = \dfrac{\tau_{k(j),i}^2}{4} \left( \| \check{x}_i^{k(j)+1} -  x_i^k\|^2 - \left\langle \sum_{j \in N(i)} W_{ij}(\check{x}_j^{k(j)+1} - x_j^k),  \check{x}_i^{k(j)+1} -  x_i^k  \right\rangle \right) \\ + \tau_{k(j),i}\left[h_i(\check{x}_i^{k(j)+1}) - h_i(x_i^k) - \langle \nabla h_i(x_i^k), \check{x}_i^{k(j)+1} - x_i^k \rangle \right] 
             - \dfrac{\deltaB+\deltaL}{2\beta}\| \check{x}_i^{k(j)+1} - x_i^k \|^2.
\end{multline*}  Note that when both linesearch procedures terminate, they imply  \begin{multline*} 
       \dfrac{\tau_{k}^2}{4}\displaystyle\sum_{i=1}^n \left( \| x_i^{k+1} -  x_i^k\|^2 - \left\langle \sum_{j \in N(i)} W_{ij}(x_j^{k+1} - x_j^k),  x_i^{k+1} -  x_i^k  \right\rangle \right) \\ + \tau_{k} \sum_{i=1}^n\left[h_i(x_i^{k+1}) - h_i(x_i^k) - \langle \nabla h_i(x_i^k), x_i^{k+1} - x_i^k \rangle \right] \leq \dfrac{\deltaB+\deltaL}{2\beta}\sum_{i=1}^n\| x_i^{k+1} - x_i^k \|^2.
    \end{multline*} These two options could  be as costly as estimating $\lambda_{\min}(W)$ separately, since they  need a number of round of network communications equal to linesearch steps, per iteration. In this sense, we have the choice of reallocating the possibly highly costly network communication rounds per iteration to a separate problem to compute $\lambda_{\min}(W)$. The modifications of the linesearch procedures \eqref{prox:k(j)}--\eqref{not-LS-i} and  \eqref{prox:k(j)-i}--\eqref{LS-i-b} by using \eqref{LS-with-W} are also well-defined, and convergence of the resulting method follows from Remark~\ref{alternative-LS}(i).

    {\item In view of Remark~\ref{r:3.5}~(i)~\&~(ii), the primal-dual gap satisfies $\hat{\mathcal{G}}_{(\hat{\bu}, \hat{\bx})}(U^k,X^{k}) = O(\frac{1}{k})$ in Corollary~\ref{c:convergence}, and $\hat{\mathcal{G}}_{(\hat{\bu}, \hat{\bx})}(U^k,X^{k}) = O(\frac{1}{k^2})$ when $f$ is strongly convex. Meanwhile, the original PG-EXTRA method \cite{shi2015proximal} with constant stepsizes and globally Lipschitz continuous gradients shows $o(\frac{1}{k})$ rate of convergence for the optimality and feasibility residuals (cf. \cite[Theorem 1]{shi2015proximal}). }
\end{enumerate}
\end{remark}

{\section{Numerical experiments} \label{s:numerical}


In this section, we present two applications in which the smooth part of the model has locally Lipschitz gradients: the Poisson regression and the covariance estimation problems. We discuss the numerical results we obtain for the distributed version of the problems for synthetic data.

\subsection{Distributed Poisson regression}

We start by describing a model for the problem of recovering an object from a set of measurements corrupted by Poisson data. Given an object $x\in\R^d$ to be reconstructed, a matrix $A \in \R^{p \times d}$ associated with a measurement procedure, and a vector $b \in \R^p$ representing background noise, for all $j = 1, \dots, d$, the $j$-th measurement  $y_j \in \Z_{+}$ is the realisation of a Poisson random variable with expected value $(Ax+b)_j$. If the underlying Poisson random variables are independently distributed, then by the maximum likelihood principle, by finding a maximizer of the likelihood \begin{equation*}
    x \mapsto \prod_{j=1}^p \dfrac{(Ax+b)_j^{y_j}\exp\big(-(Ax+b)_j\big)}{y_j!},
\end{equation*} subject to $x$ belonging to the nonnegative orthant, we can recover the object $x$ from the measurements $y_j$, $j = 1, \dots, d$ (see, e.g., \cite[Section 3.1]{bertero2009image}). Observe that the negative log-likelihood, up to constant terms, is given by \begin{equation*}
    h(x) = \text{KL}(Ax+b,y),
\end{equation*} where KL is defined in \eqref{e:KL}.

In the context of distributed optimisation, the multiple image deconvolution problem \cite{bertero2000application} is an extension of the problem described above: the goal is to recover an object from different sets of measurements with different reconstruction matrices. More precisely, consider $n$ agents over a network, and for each $i=1,\dots,n$, consider  matrices $A^{(i)} \in \R^{p \times d}$ associated with agent $i$ measurement procedure, and background noise vectors $b^{(i)} \in \R^p$. For each $j = 1, \dots, d$, each $j$-th measurement $y^{(i)}_j \in \Z_{+}$ of agent $i$, is the realisation of a Poisson random variable with expected value $(A^{(i)} x+b^{(i)})_j$. Therefore, if the underlying random processes are independently distributed, the maximum likelihood estimation problem for finding an object $x \in \R^d_+$ is equivalent to solving  
\begin{equation} \label{eq:dist-KL-problem}
    \min_{x\in \R^d} \sum_{i=1}^n i_{\R^d_+}(x) + h_i(x).
\end{equation} where \begin{equation*}
    h_i(x) := \text{KL}(A^{(i)} x+b^{(i)},y^{(i)}).
\end{equation*} Observe that problem \eqref{eq:dist-KL-problem} follows the  structure of problem \eqref{primal-problem}. Recall that we are interested in the setting where each agent $i$  only has access to $(A^{(i)}, b^{(i)}, y^{(i)})$ and can only communicate with its direct neighbors in the network. 
Therefore, an application of Algorithm~\ref{a:PGE} to solve problem \eqref{eq:dist-KL-problem}, barring the specification of the linesearch procedure, has the following as proximal step: in iteration $k \geq 1$, for each agent $i = 1,\dots, n$, \[
x^{k+1}_i = \max\big(0,x^k_i- \beta\tau_k[\overline{u}^k_i + \nabla h_i(x^k_i)]\big),
\] 
where the maximum in the last line is understood element-wise, and $$\nabla h_i(x) = e - (A^{(i)})^\top \big( b^{(i)} \oslash (A^{(i)} x +b^{(i)}) \big).$$

\noindent \textbf{Setting 1}: in the first experiment, we take $\beta = 2.0$ to enhance proximal steps, $\delta_L = 0.5$, $\delta_K = 0.9999 - \delta_L$, $\rho = 0.95$,  $\gamma = 0.99$. The starting point is set to be the empirical approximation given by noisy measurements each agent has locally available.  We also set  $n=4$ and $d= 64 \times 64$ (size of the images). The sample image was generated using Claude AI \cite{claude2025}.

We start showing the results obtained for a ring network configuration. In Figure~\ref{fig:5.1-poisson-1}, we show the reconstructions of a simple image obtained by PG-EXTRA \cite{shi2015proximal} with constant stepsize, Algorithm~\ref{a:PGE} with constant stepsize, and Algorithm~\ref{a:PGE} implemented with Subroutines~\ref{LS:1} and~\ref{LS:2}.  Choosing an appropriate stepsize to run PG-EXTRA in this setting, in which the gradient of the smooth term of the objective function is not globally Lipschitz continuous is not a straightforward task. We show the results of the best tested stepsize. For Algorithm~\ref{a:PGE}, we choose $\tau_0 = 0.1\cdot\frac{\sqrt{2\deltaB}}{\sqrt{\beta(1 - \lambda_{\min}(W))}}$, inspired by the initialisation in \eqref{init-LS}.

Observe that, from Figure~\ref{fig:5.1-poisson-1}, PG-EXTRA shows a similar behaviour to Algorithm~\ref{a:PGE} implemented to any of the two linesearch subroutines, although the resulting reconstruction is more blurry. Algorithm~\ref{a:PGE} with constant stepsize has better defined edges for the square shape, although the quality of the reconstruction of the circle is worse compared to Algorithm~\ref{a:PGE} implemented to any of the two linesearch subroutines. In any case, the shape of the reconstructions given by Algorithm~\ref{a:PGE} resembles the (unique) original image, despite the fact that the noisy measurements get progressively worse from agent $1$ to agent $n$.

\begin{figure}
    \centering
    \includegraphics[width=\linewidth]{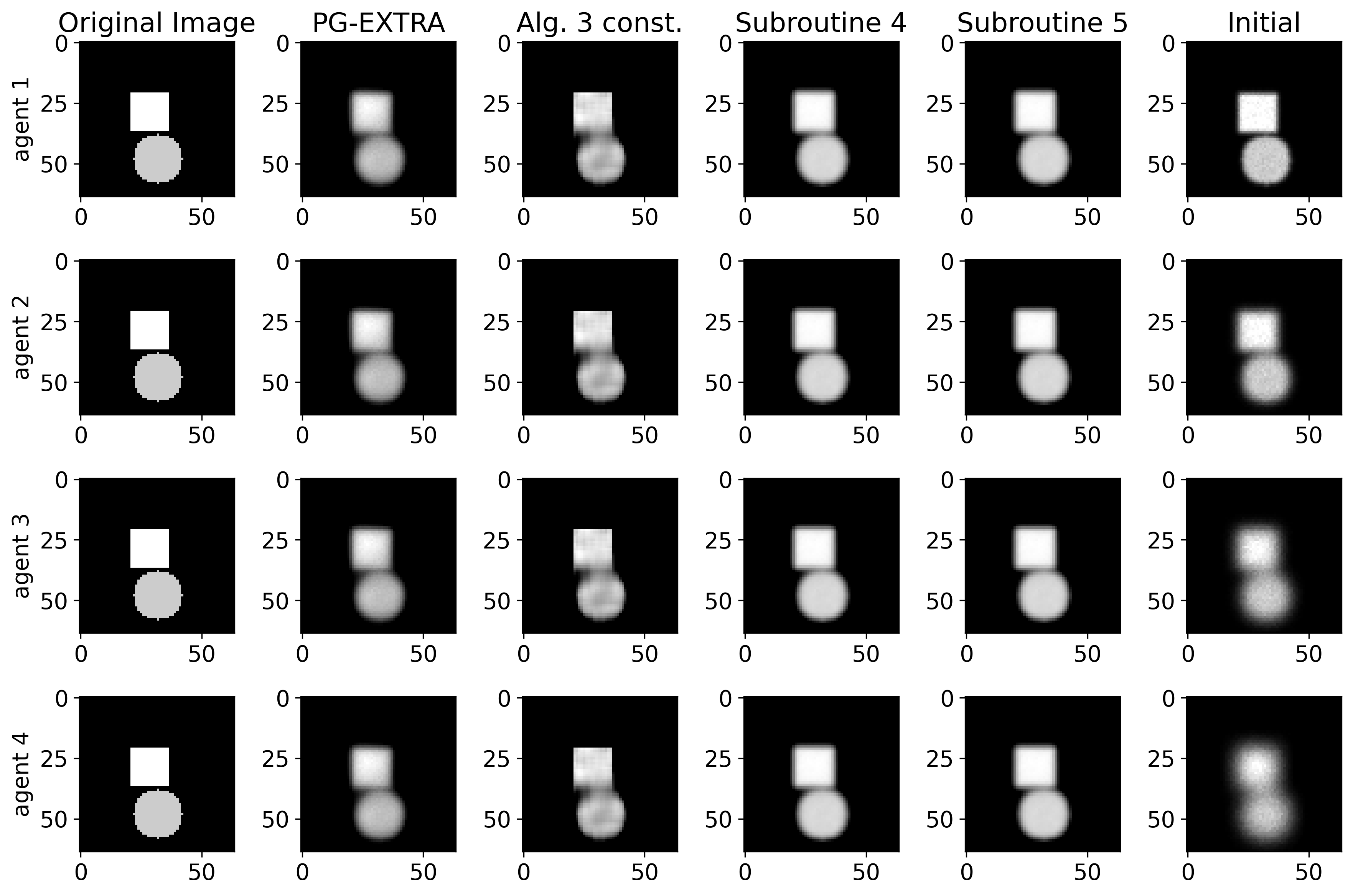}
    \caption{Reconstruction of a simple image from distributed blurry and noisy measurements. Each row corresponds to an agent, and each column to a method. From left to right: the unique original image, PG-EXTRA \cite{shi2015proximal} with constant stepsize $\tau = 0.01$, Algorithm~\ref{a:PGE} with constant stepsize $\tau_0 = 0.1\cdot\frac{\sqrt{2\deltaB}}{\sqrt{\beta(1 - \lambda_{\min}(W))}}$, Algorithm~\ref{a:PGE} implemented with Subroutine~\ref{LS:1} starting from $\tau = \tau_0$, and Algorithm~\ref{a:PGE} implemented with Subroutine~\ref{LS:2} starting from $\tau = \tau_0$.}
    \label{fig:5.1-poisson-1}
\end{figure}

\noindent \textbf{Setting 2}: in order to induce more regularity to the reconstruction, we include the strongly convex $\ell_2$ penalization term. We replace the indicator of the nonnegative orthant with $f_i(x_i) = \iota_{\R^d_+}(x_i) + \frac{\lambda}{2}\|x_i\|_2^2$, changing the $x$-update to \begin{equation*}
    x^{k+1}_i  = \max\left(0,\frac{x^k_i- \beta\tau_k[\overline{u}^k_i + \nabla h_i(x^k_i)]}{1+\lambda\tau_k}\right).
\end{equation*} We keep the same parameters as in Setting 1, changing only the image to reconstruct, also generated using Claude AI \cite{claude2025}. We set the regularisation term   $\lambda = 0.001$.

Figure~\ref{fig:5.1-poisson-2} shows that PG-EXTRA with constant stepsize produces a similar reconstruction to Algorithm~\ref{a:PGE} implemented to any of the two linesearch subroutines, but still more blurry. However, Algorithm~\ref{a:PGE} with constant stepsize reconstruct an image considerably worse than the others. By obverving the columns of Subroutine~\ref{LS:1} and Subroutine~\ref{LS:2}, we see that we stopped the iterations before consensus among agents is achieved, and therefore we proceed to present the progress of the relative error (Figure~\ref{fig:5.1-img-re-2}), the feasibility measure (Figure~\ref{fig:5.1-img-feas-2}) and the stepsize (Figure~\ref{fig:5.1-img-step-2}) along outer iterations. In other words, we show the \emph{true} values of each measure, ignoring for now the corresponding values computed within the linesearch. From Figure~\ref{fig:5.1-img-re-2}, we can see that although our method quickly achieves a small relative error, similarly to the stepsizes in Figure~\ref{fig:5.1-img-step-2}, the values of $\tau_k$ need to increase in order to improve consensus among agents in Figure~\ref{fig:5.1-img-feas-2}. Despite what Figure~\ref{fig:5.1-img-step-2b} might suggest, the stepsize does not converge to $0$, as the smallest order of magnitude it achieves, $10^{-5}$, is still larger than the smallest order of magnitude achieved for the relative error, $10^{-6}$.

\begin{figure}
    \centering
    \includegraphics[width=\linewidth]{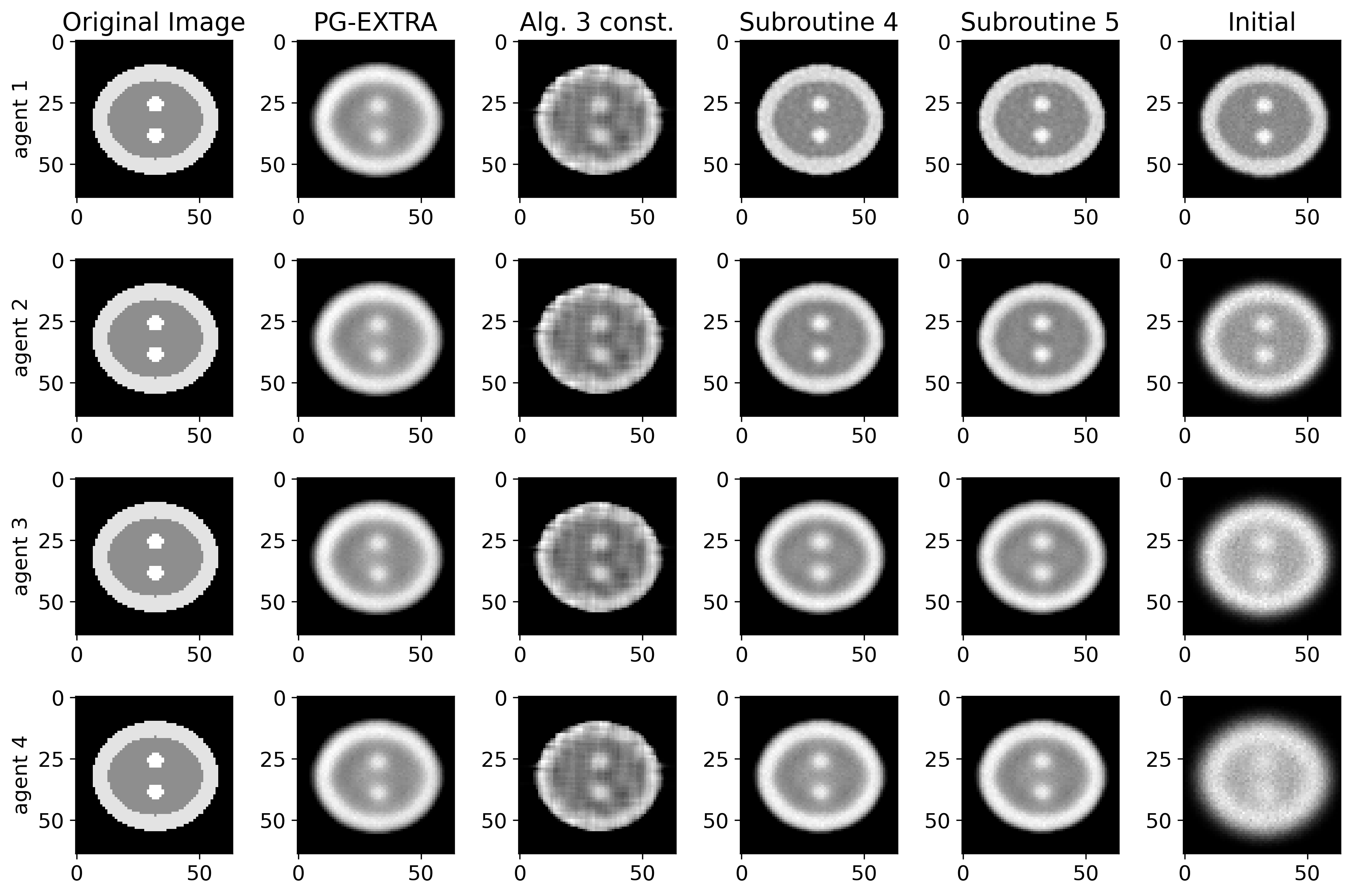}
    \caption{Reconstruction of a synthetic phantom image from distributed blurry and noisy measurements. Each row corresponds to an agent, and each column to a method. From left to right: original image, PG-EXTRA \cite{shi2015proximal} with constant stepsize $\tau = 0.01$, Algorithm~\ref{a:PGE} with constant stepsize $\tau_0 = 0.1\cdot\frac{\sqrt{2\deltaB}}{\sqrt{\beta(1 - \lambda_{\min}(W))}}$, Algorithm~\ref{a:PGE} implemented with Subroutine~\ref{LS:1} starting from $\tau = \tau_0$, and Algorithm~\ref{a:PGE} implemented with Subroutine~\ref{LS:2} starting from $\tau = \tau_0$.}
    \label{fig:5.1-poisson-2}
\end{figure}


\begin{figure}
\centering
\begin{subfigure}{0.4\textwidth}
    \includegraphics[width=\textwidth]{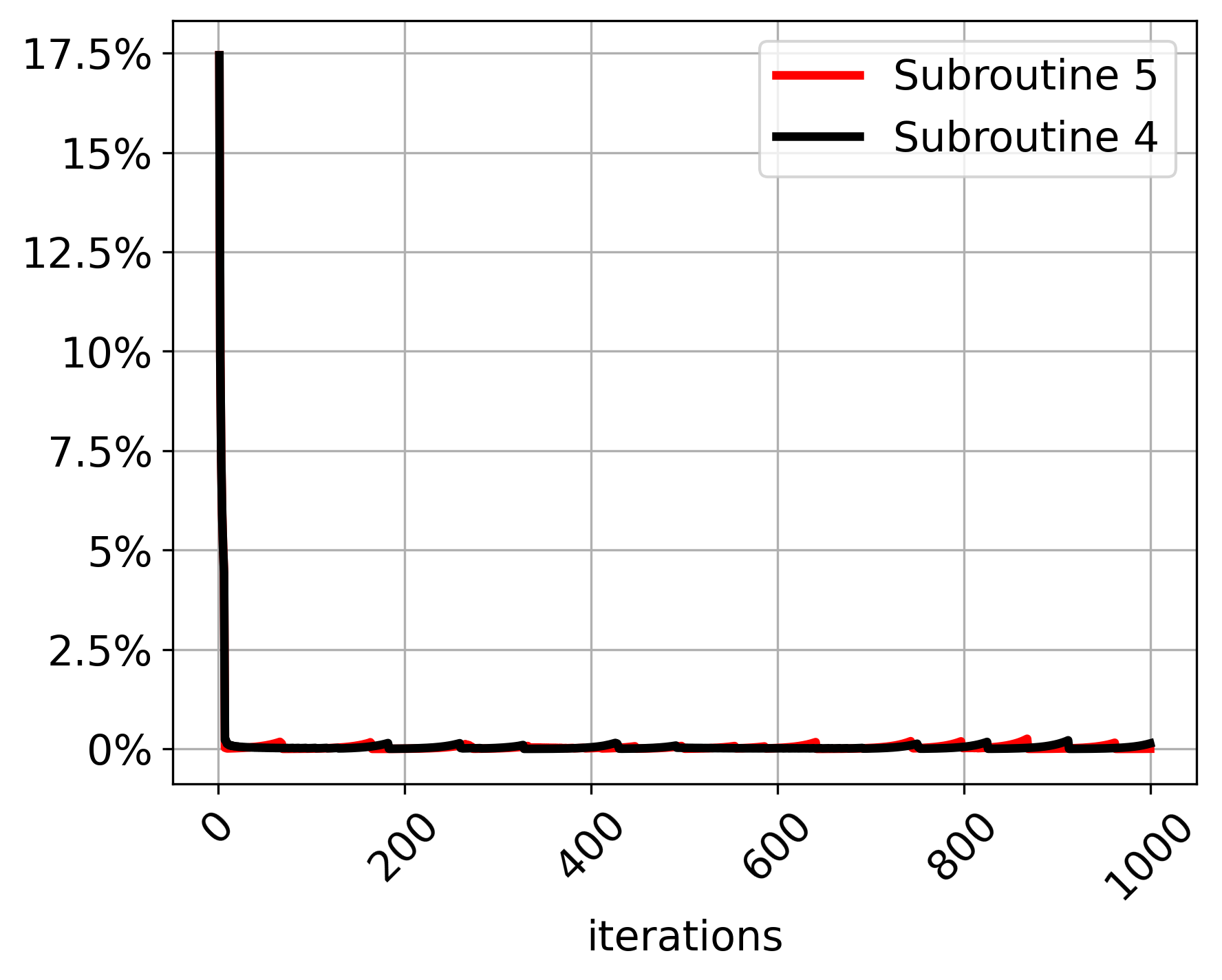}
    \caption{All outer iterations.\newline}
    \label{fig:5.1-img-re-2a}
\end{subfigure}
\hfill
\begin{subfigure}{0.4\textwidth}
    \includegraphics[width=\textwidth]{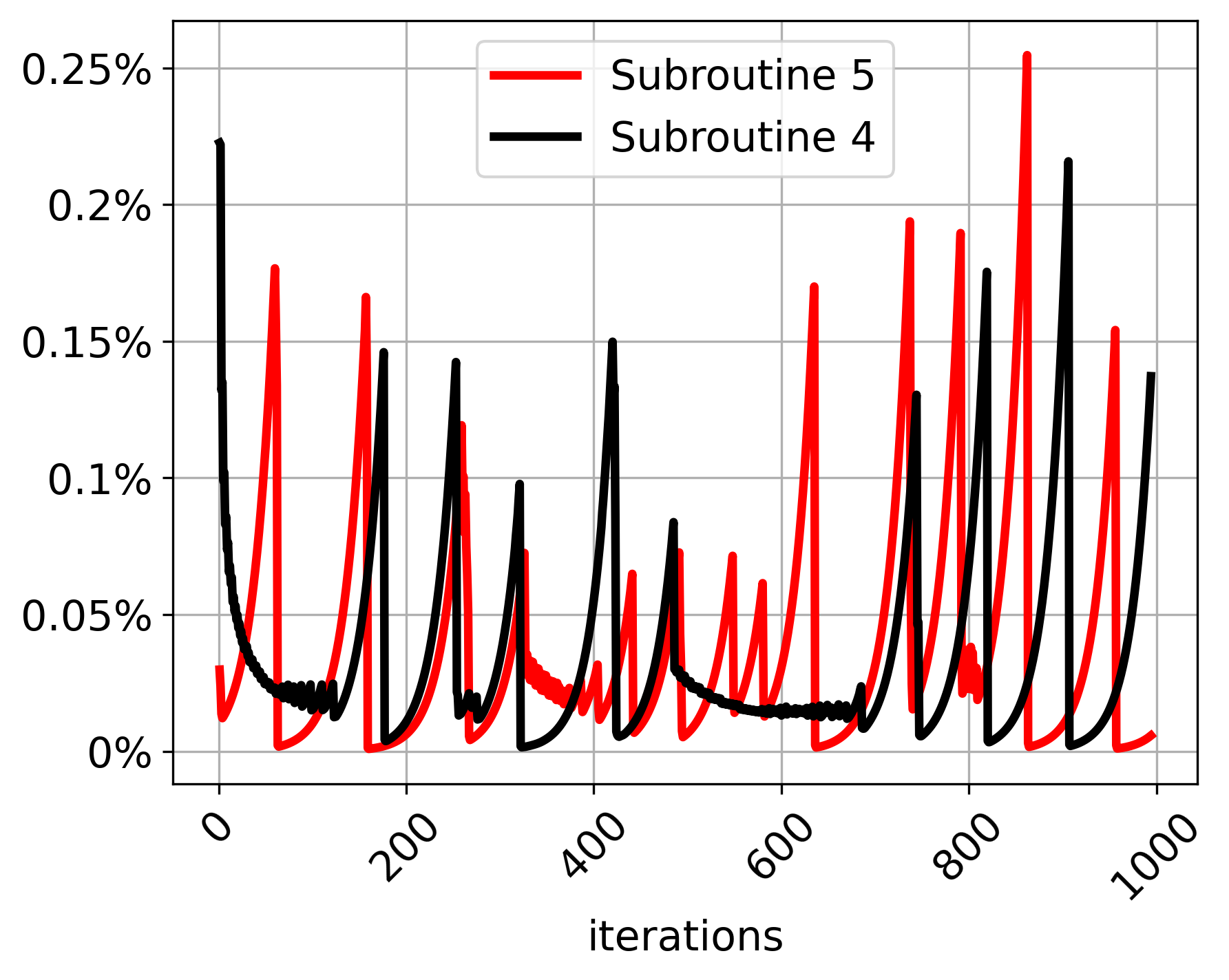}
    \caption{Outer iterations with same order of magnitude for $\Delta \bx^k$.}
    \label{fig:5.1-img-re-2b} 
\end{subfigure} 
\caption{Relative error $ \Delta \bx^k = \frac{\|\bx^{k+1}-\bx^k\|}{\|\bx^0 - \bx_{\text{true}}\|}$, where $\bx_{\text{true}}$ is the original image.}
\label{fig:5.1-img-re-2}
\end{figure}


\begin{figure}
\centering
    \includegraphics[width=0.4\textwidth]{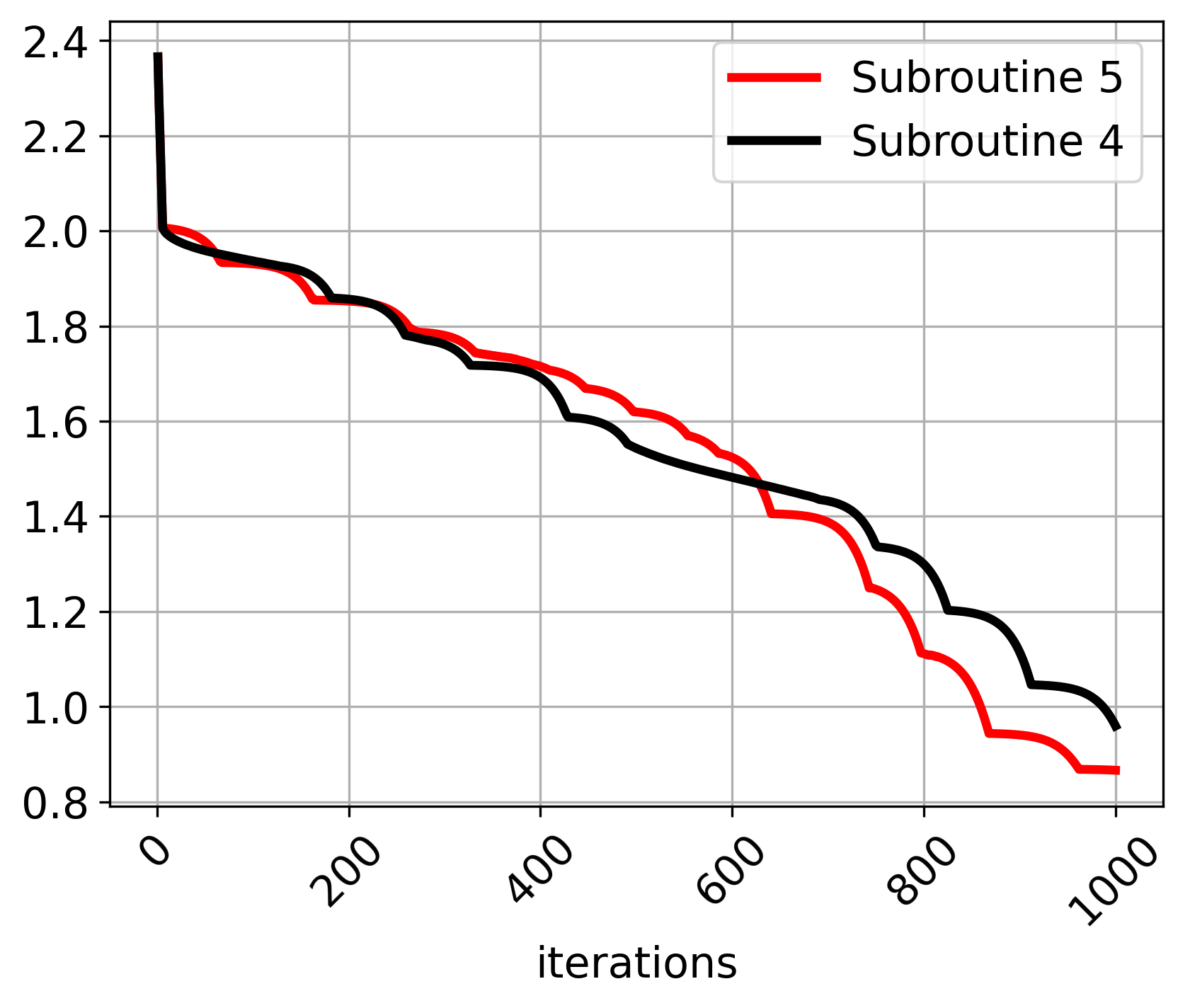}
\caption{Feasibility measure $\|(I-W)\bx^k\|$ along outer iterations.}
\label{fig:5.1-img-feas-2}
\end{figure}


\begin{figure}
\centering
\begin{subfigure}{0.4\textwidth}
    \includegraphics[width=\textwidth]{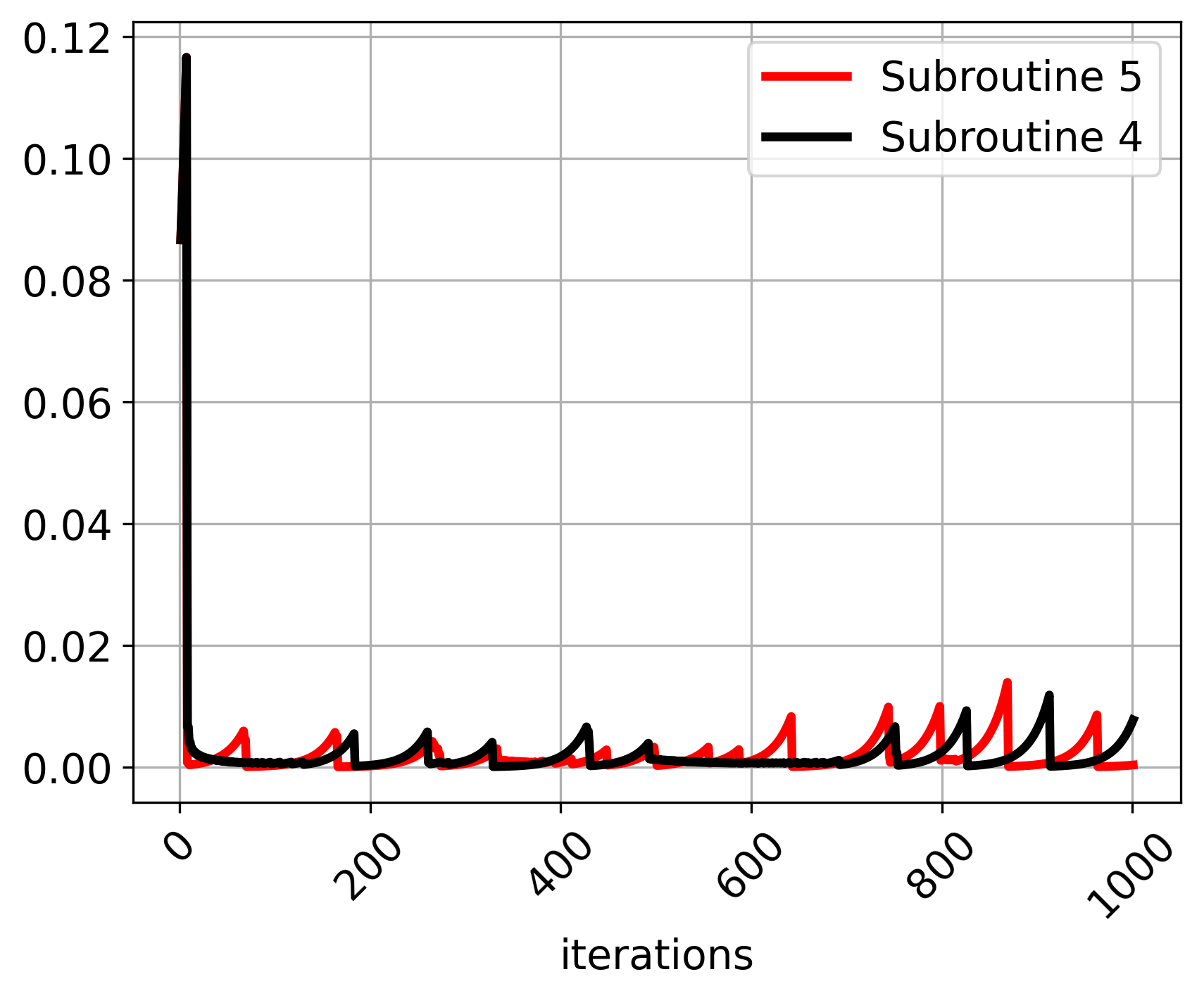}
    \caption{All outer iterations.\newline}
    \label{fig:5.1-img-step-2a}
\end{subfigure}
\hfill
\begin{subfigure}{0.4\textwidth}
    \includegraphics[width=\textwidth]{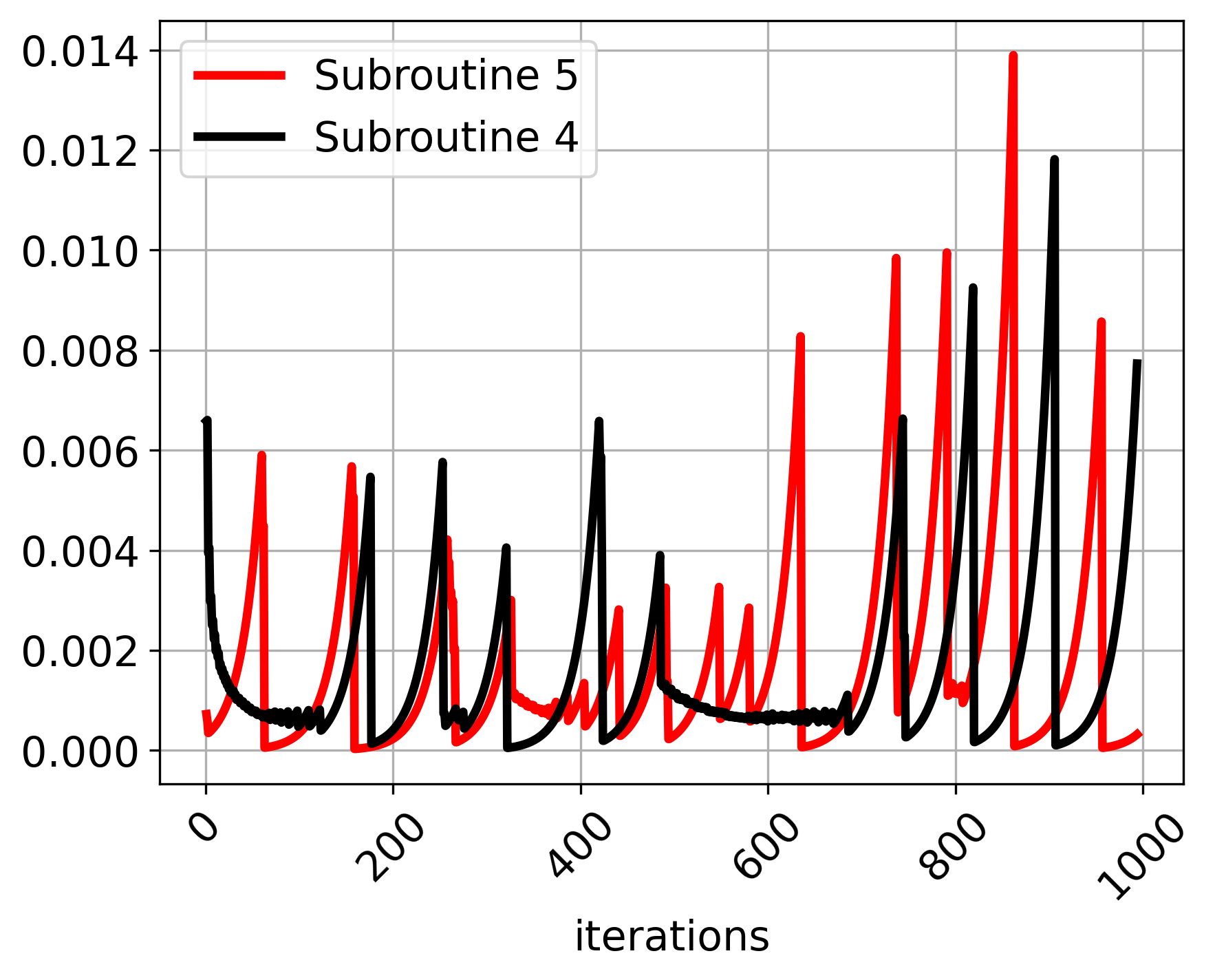}
    \caption{Outer iterations with same order of magnitude for stepsize.}
    \label{fig:5.1-img-step-2b} 
\end{subfigure} 
\caption{Stepsize sequence alongside outer iterations.}
\label{fig:5.1-img-step-2}
\end{figure}

\subsection{Distributed maximum likelihood estimate of the information matrix}

 The second application we study is the estimation of the inverse of a covariance matrix. Given a Gaussian random variable $y \in \R^d$ with mean zero, we consider the problem of estimating its covariance matrix $\Sigma = \mathbb{E}yy^{\top}$ using $M \geq 1$ samples $\{y_j \in \R^d : j = 1, \dots, M \}$ of $y$. In order to achieve this goal, we solve the following convex optimisation problem \cite[Eq. (7.5)]{boyd2004convex} for the information matrix $X$ (the inverse of the covariance matrix): \begin{equation*}
    \max_{X \in \mathbb{S}^d} \log \det(X) - \tr(XY) \text{~~s.t.~~} X \in C,
\end{equation*} where $\mathbb{S}^d$ denotes the space of  symmetric matrices in $\R^{d\times d}$, \begin{equation} \label{5.2:empirical} Y = \frac{1}{M}\sum_{j=1}^M y_j y_j^{\top}\end{equation} is the sample covariance, and $C \subseteq \mathbb{S}^d$ is a constraint set used to encode prior knowledge of the solution. Following \cite[Eq. (45)]{malitsky2024adaptive}, we consider $$C = \{X \in \mathbb{S}^d: l I \preccurlyeq X \preccurlyeq u I \}$$ for some real values $0 < l \leq u$.

In distributed settings, we are interested in the case when the agents in a network do not have access to all the samples, but rather to a subset of them. More specifically, suppose  $\{1, \dots, M\}$ is partitioned into $n$ sets $S_i$ of size $|S_i|$ for agent $i = 1, \dots, n$. Each agent $i$ only has access to the samples $\{y_j: j \in S_i\}$. In this manner, the sample covariance can be written as follows \begin{equation*}
     Y = \frac{1}{M}\sum_{i=1}^n \sum_{j \in S_i} y_j y_j^{\top} = \frac{1}{M}\sum_{i=1}^n |S_i| Y_i, \end{equation*} where $Y_i = \frac{1}{|S_i|}\sum_{j \in S_i} y_j y_j^{\top}$ for $i = 1,\dots, n$. Hence, the distributed maximum likelihood estimate of the information matrix problem can be cast in the form of problem \eqref{primal-problem} as follows: \begin{equation} \label{eq:dist-MLEI}
    \min_{X \in \mathbb{S}^d}\sum_{i=1}^n h_i(X) + \iota_C(X),
\end{equation} with\begin{equation*}
    h_i(X) = -|S_i| \big( \log\det(X) - \tr(XY_i)\big).
\end{equation*} Denote by $X_i$ a copy of the decision variable $X$ owned by agent $i$. An application of Algorithm~\ref{a:PGE} to solve problem \eqref{eq:dist-MLEI}, barring the specification of the linesearch procedure, yields the following iteration scheme: in iteration $k \geq 1$, for each agent $i = 1,\dots, n$, \begin{equation*} 
    \left\{\begin{aligned}
         V^k_i& = V^{k-1}_i+ \frac{\tau_{k-1}}{2}\sum_{t=1}^n (I-W)_{it}X_t  \\
         \overline{V}^k_i &= V^k_i + \theta_k(V^k_i - V^{k-1}_i)\\
         X^{k+1}_i & = \proj_C(X^k_i- \beta\tau_k[\overline{V}^k_i + \nabla h_i(X^k_i)]) 
    \end{aligned}\right.
\end{equation*}  where $$\nabla h_i(X) = -|S_i|\big( X^{-1} - Y_i \big).$$

\noindent \textbf{Setting}: in the following experiment, we take $\beta = 1.0$, $\delta_L = 0.5$, $\delta_K = 0.9999 - \delta_L$, $\rho = 0.95$,  $\gamma = 0.99$, and for all $i$, $x_i^1 = I_{d\times d}$. We also take $l = 0.7$ and $u = 1.8$, defining a box that contains the solution. We also set  $n=10$, $d= 5$, and $|S_i| = 1$ for all $i = 1, \dots, n$. In our simulations, data was generated following the sparse covariance estimation example from the scikit-learn documentation \cite{scikit-learn}: given a sparse positive definite matrix $P$, the target covariance matrix is given by $\Sigma = P^{-1}$. 

We start showing the results obtained for a ring network configuration. Figure~\ref{fig:5.2-info-matrix-1} shows the inverse of the information matrices given by the initial empirical data, and Algorithm~\ref{a:PGE} implemented with constant stepsize, and Subroutines~\ref{LS:1} and~\ref{LS:2}. Note that Algorithm~\ref{a:PGE} with linesearch identifies the entries in the diagonal, different from the case when a constant stepsize is chosen.

\begin{figure}
    \centering
    \includegraphics[width=0.8\linewidth]{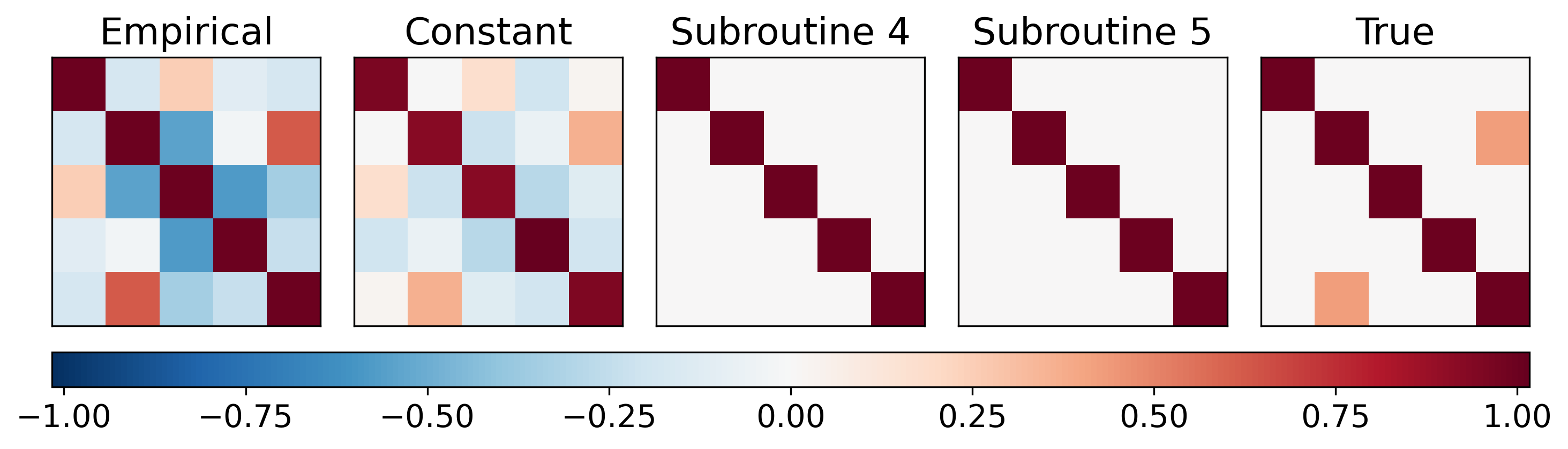}
    \caption{Covariance estimated by empirical data ($Y$ given in \eqref{5.2:empirical}), Algorithm~\ref{a:PGE} implemented with constant stepsize, and Subroutines~\ref{LS:1} and~\ref{LS:2}, compared to the true covariance. Initial stepsize: $\tau_0 = \frac{\sqrt{2\deltaB}}{\sqrt{\beta(1 - \lambda_{\min}(W))}}$.}
    \label{fig:5.2-info-matrix-1}
\end{figure}

In this example, we also include the effect of the inner iterations of the linesearch procedures on the performance measures. We do not take into account the cost of rounds of communications, as our experiments only simulate a distributed network, and therefore they do not represent a real cost. In order to simplify the comparison of the numerical results, suppose without loss of generality that the agents perform proximal-gradient evaluations synchronously. In other words, each round of proximal-gradient evaluations consists of each agent performing its own local proximal-gradient evaluation plus some possible idle time. The performance measures are  thus shown as a function of the proximal-gradient evaluation (synchronous) rounds. From construction, for Subroutine~\ref{LS:2}, some agents may need more inner linesearch steps to find an appropriate stepsize, and thus a round of proximal-gradient evaluations could consist of at least one agent performing one linesearch inner step, while the rest are idle.

Furthermore, in order to make a fair comparison, we show the results of the mean value and standard deviation across agents. More specifically, for Subroutine~\ref{LS:1}, it is possible to measure the progress $(\bx^k)_k$  in each linesearch step, since all the agents synchronously perform one proximal-gradient step, and then the  round of communications used to test the linesearch condition can be used to also evaluate the performance measures. However,  for Subroutine~\ref{LS:2}, the same cannot be done, as the round of communications used to compute the minimum stepsize among agents is only performed after each agent has finished its own linesearch inner loop. Therefore, every time a linesearch step is performed for each agent, an extra round of communication would be needed only to evaluate the performance measures, dramatically increasing the computational costs. Instead, we keep track of the local performance measures of interest for each agent separately,  for both subroutines.

In Figures~\ref{fig:5.2-rel-1} and \ref{fig:5.2-step-1}, we compare the results obtained for the relative error and the stepsizes, respectively, alongside  proximal-gradient evaluation rounds. We use the latter to measure the performance of both Subroutines~\ref{LS:1} and ~\ref{LS:2}, as each outer iteration hides, in principle, an unknown number of linesearch inner iterations. Figure~\ref{fig:5.2-rel-1} shows that both subroutines stabilise at around 17.5\% of relative error, although Subroutine~\ref{LS:1} takes fewer proximal-gradient evaluations to satisfy the stopping test: $\max\{\|\bx^{k+1} - \bx^k\|, \|(I-W)\bx^k\| \} < 10^{-3}$. Figure~\ref{fig:5.2-step-1} shows that both subroutines eventually find stepsizes in a given neighbourhood, and stay bounded away from $0$.

We also observe that both subroutines keep the iterates (numerically) feasible, that is, agents might  deviate from consensus alongside iterations only marginally, and thus we do not report it.


\begin{figure}
\centering
\begin{subfigure}{0.32\textwidth}
    \includegraphics[width=\textwidth]{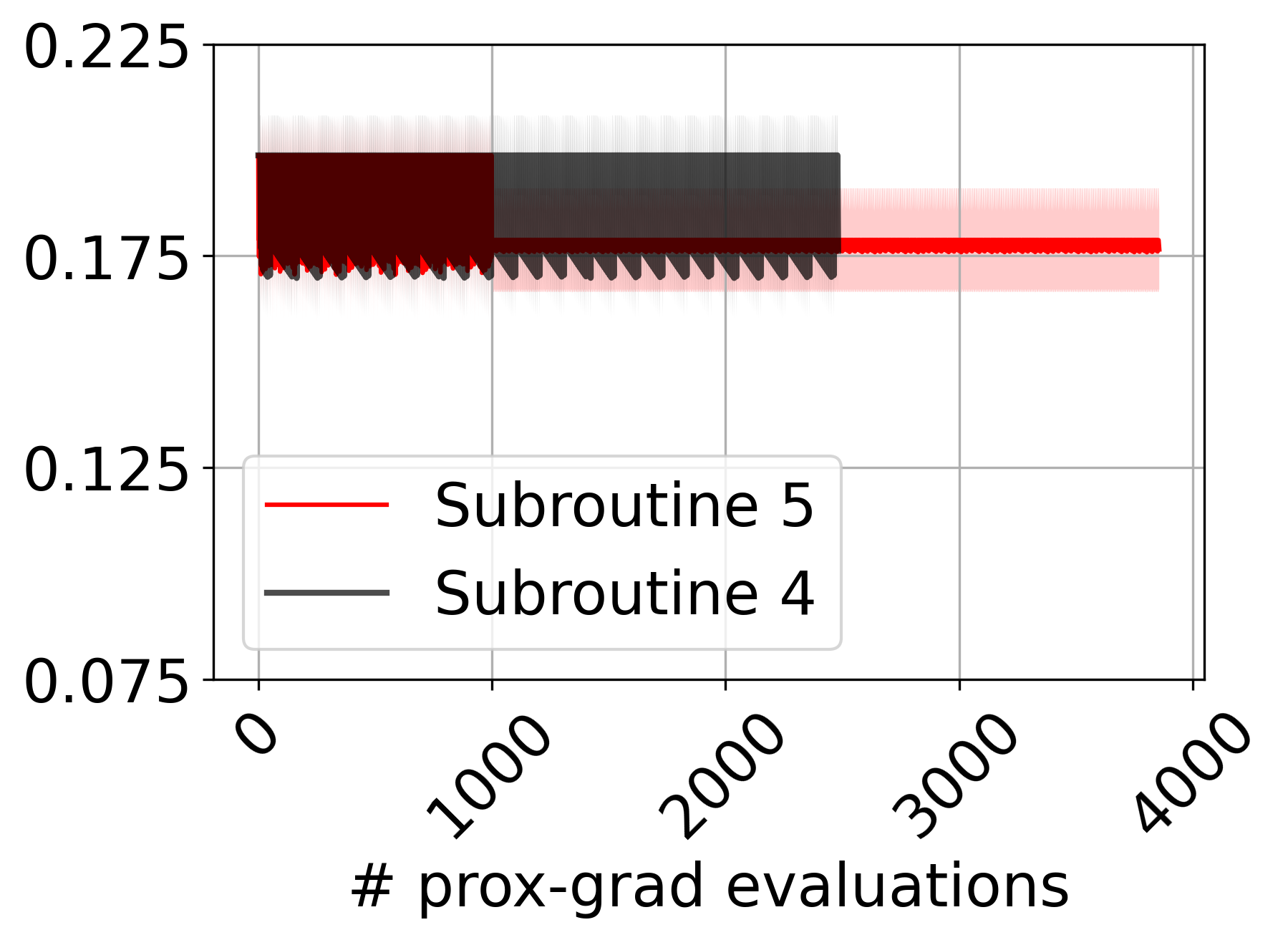}
    \caption{$\tau_0 = \frac{\sqrt{2\deltaB}}{\sqrt{\beta(1 - \lambda_{\min}(W))}}$}
    \label{fig:5.2-info-matrix-rel-t1}
\end{subfigure}
\hfill
\begin{subfigure}{0.32\textwidth}
    \includegraphics[width=\textwidth]{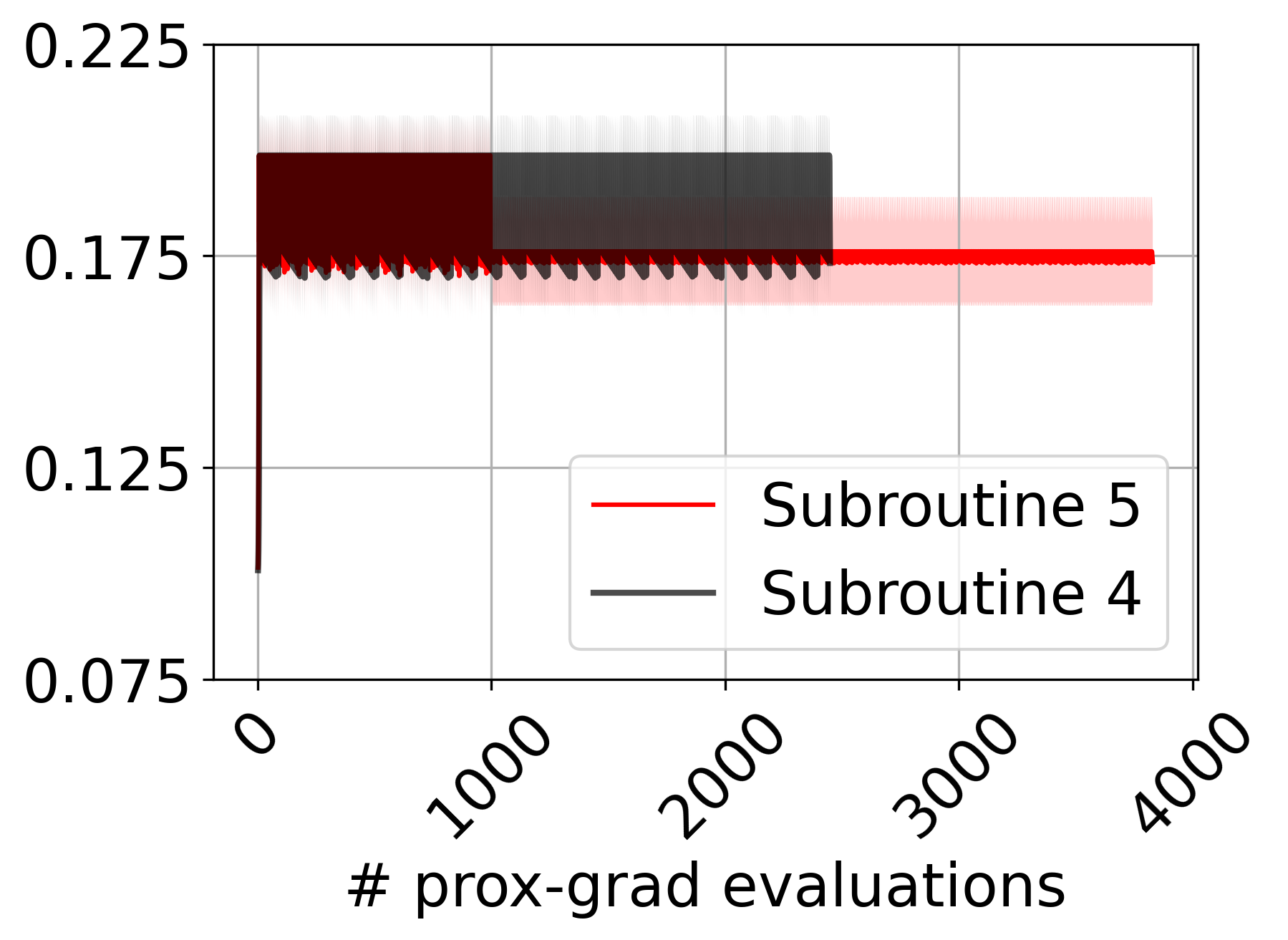}
    \caption{$\tau_0 = 0.1\cdot\frac{\sqrt{2\deltaB}}{\sqrt{\beta(1 - \lambda_{\min}(W))}}$}
    \label{fig:5.2-info-matrix-rel-t2}
\end{subfigure} 
\hfill
\begin{subfigure}{0.32\textwidth}
    \includegraphics[width=\textwidth]{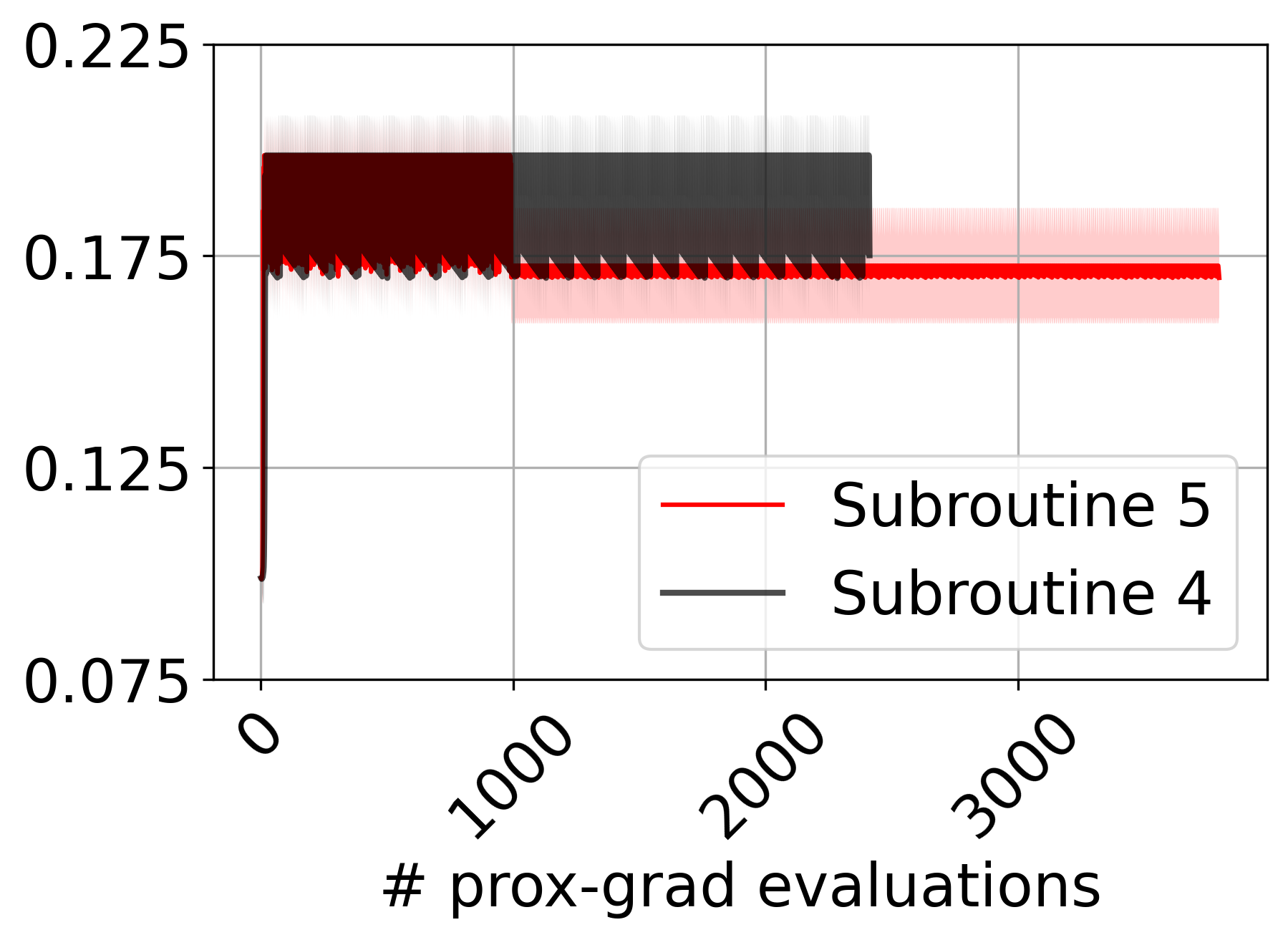}
    \caption{$\tau_0 = 0.01\cdot\frac{\sqrt{2\deltaB}}{\sqrt{\beta(1 - \lambda_{\min}(W))}}$}
    \label{fig:5.2-info-matrix-rel-t3}
\end{subfigure} 
\caption{Mean relative error $\frac{\|x_i^k - x^*\|}{\|x_i^0 - x^*\|}$ and standard deviation across agents as a function of the number of proximal-gradient evaluations, for different initial stepsizes.}
\label{fig:5.2-rel-1}
\end{figure}


\begin{figure}
\centering
\begin{subfigure}{0.32\textwidth}
    \includegraphics[width=\textwidth]{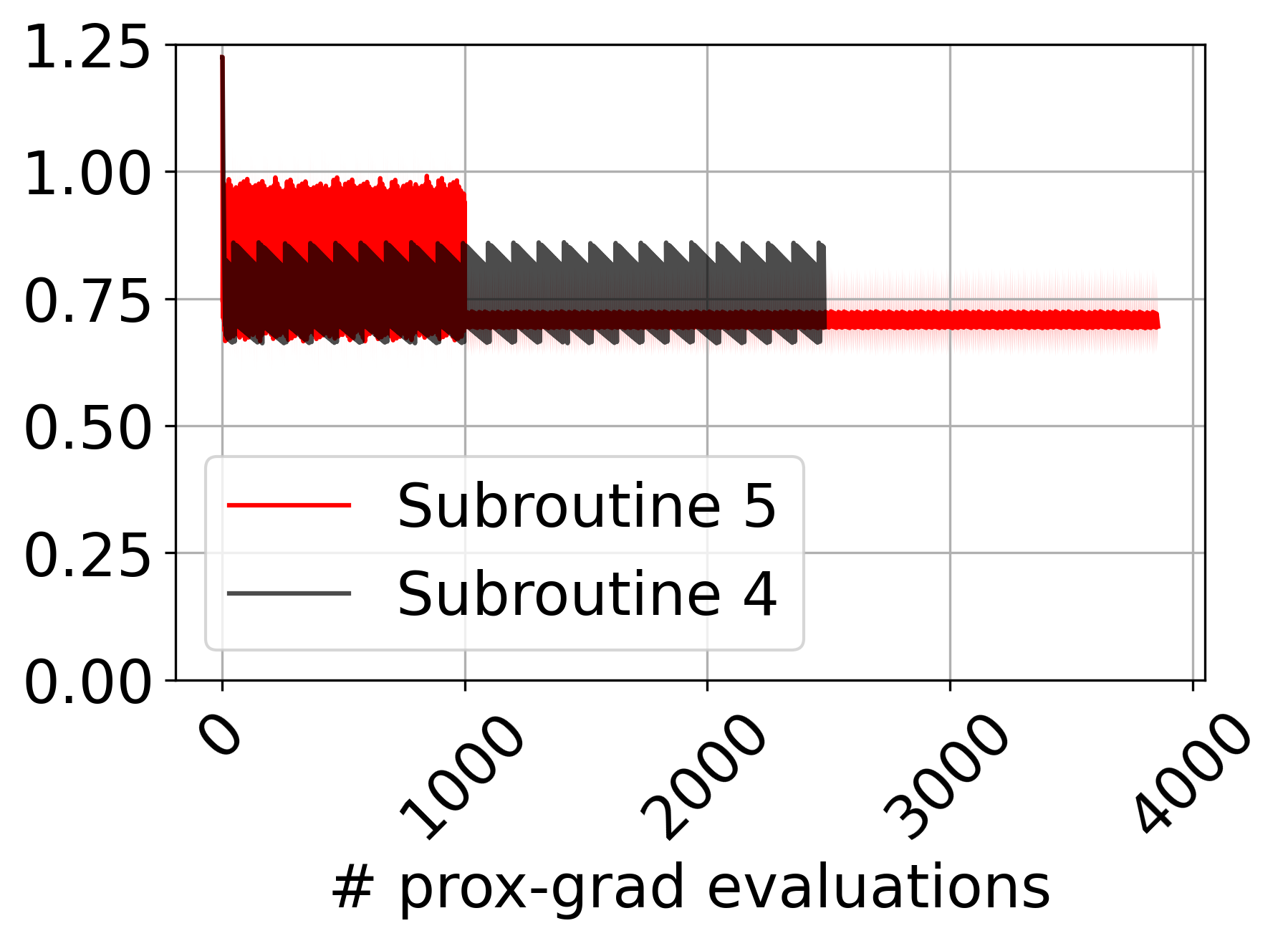}
    \caption{$\tau_0 = \frac{\sqrt{2\deltaB}}{\sqrt{\beta(1 - \lambda_{\min}(W))}}$}
    \label{fig:5.2-info-matrix-step-t1}
\end{subfigure}
\hfill
\begin{subfigure}{0.32\textwidth}
    \includegraphics[width=\textwidth]{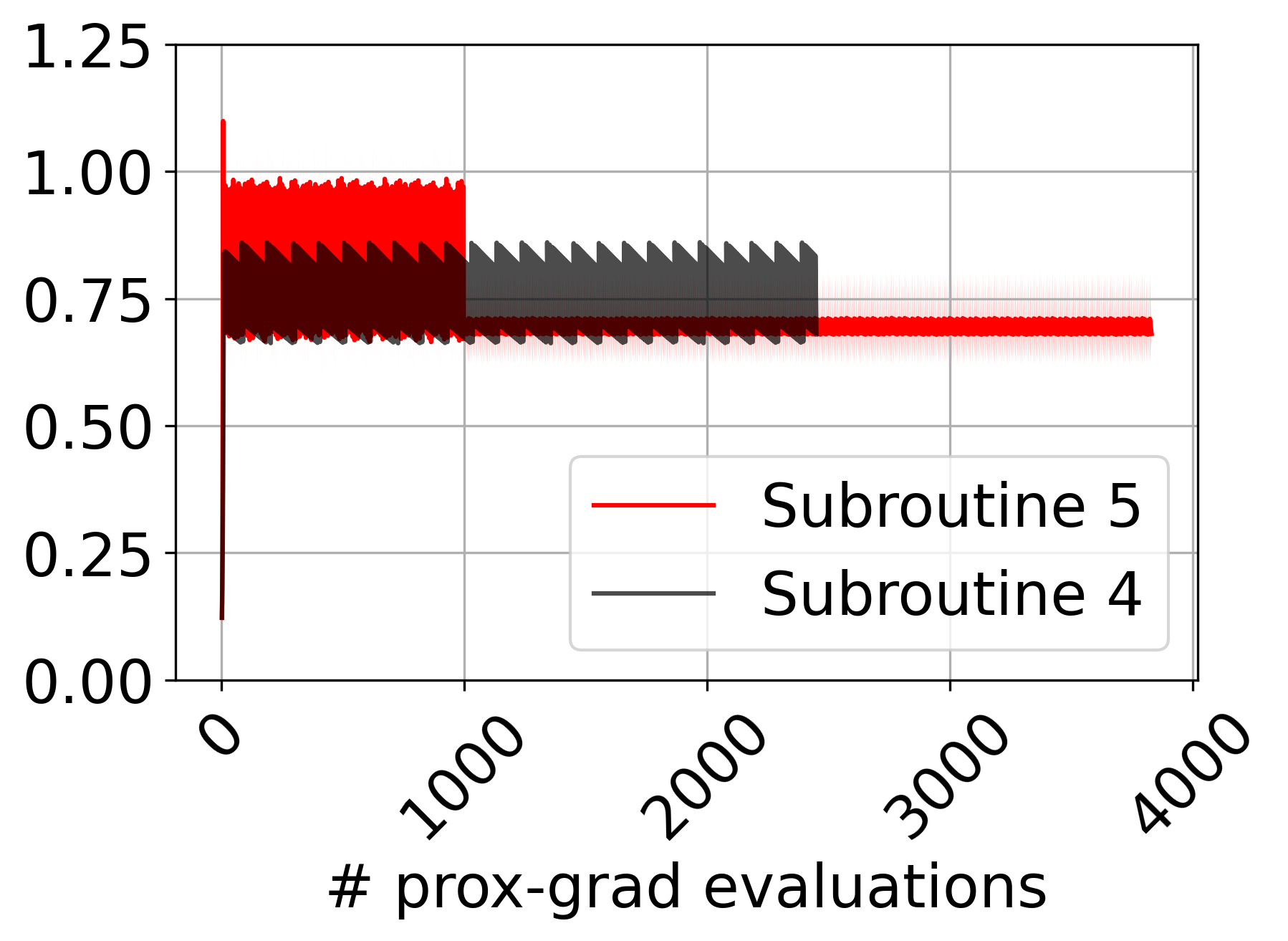}
    \caption{$\tau_0 = 0.1\cdot\frac{\sqrt{2\deltaB}}{\sqrt{\beta(1 - \lambda_{\min}(W))}}$}
    \label{fig:5.2-info-matrix-step-t2}
\end{subfigure} 
\hfill
\begin{subfigure}{0.32\textwidth}
    \includegraphics[width=\textwidth]{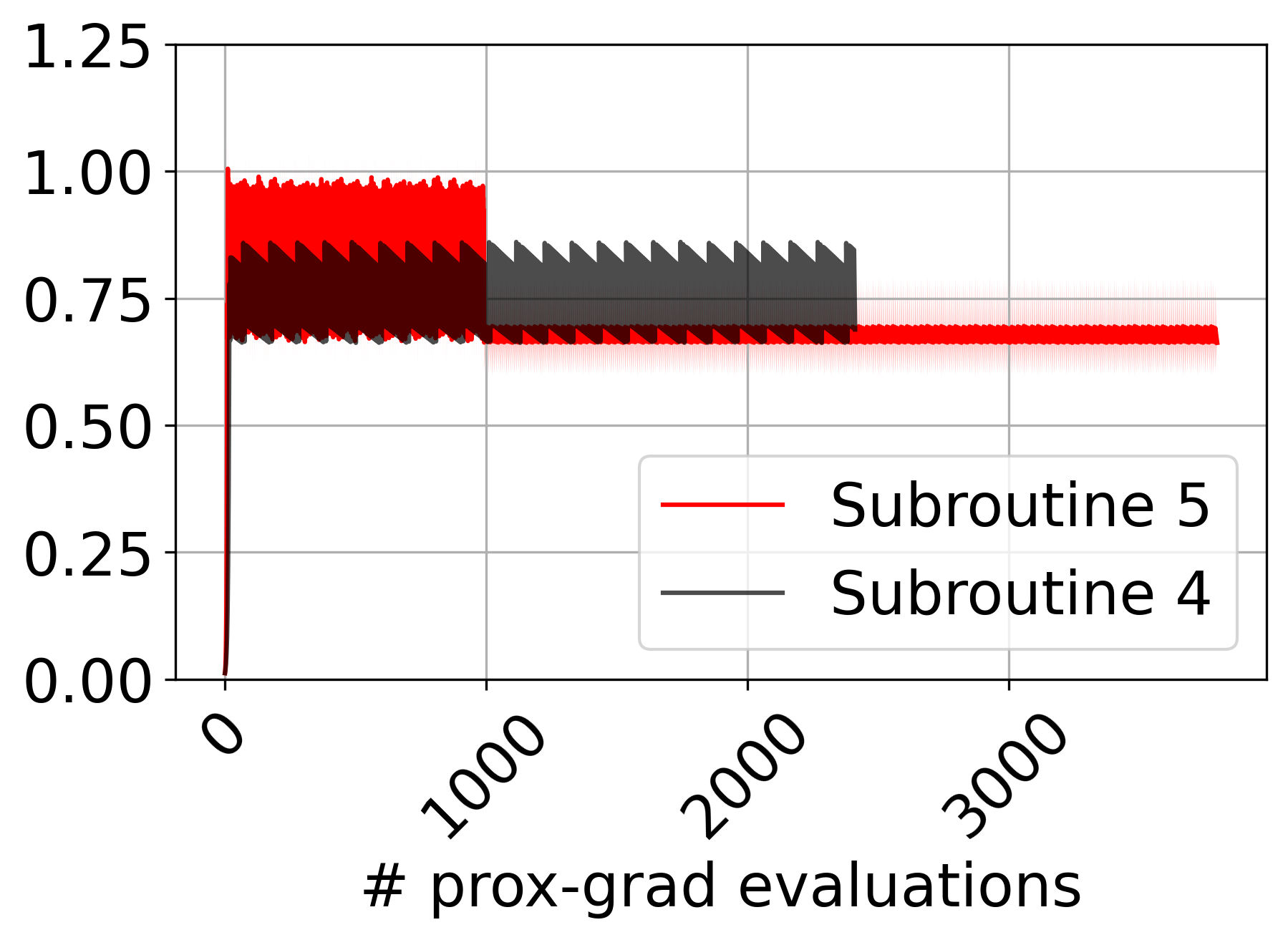}
    \caption{$\tau_0 = 0.01\cdot\frac{\sqrt{2\deltaB}}{\sqrt{\beta(1 - \lambda_{\min}(W))}}$}
    \label{fig:5.2-info-matrix-step-t3}
\end{subfigure} 
\caption{(Mean) stepsize and standard deviation across agents as a function of the number of proximal-gradient evaluations, for different initial stepsizes.}
\label{fig:5.2-step-1}
\end{figure}
}

\section{Concluding remarks}
\label{s:conclusion}

In this work, we propose an abstract linesearch framework for a primal-dual splitting method to solve min-max convex concave problems, that can be viewed as a generalisation of the method proposed in 
 \cite{malitsky2018first}, and we provide an implementation of such a linesearch that naturally satisfies the assumptions. In the distributed optimisation setting, we propose two distributed linesearch procedures, in such a way that the resulting method extends PG-EXTRA defining variable stepsizes. We also allow the gradients $\nabla h_i$ to be locally Lipschitz continuous for each agent $i \in \{1, \dots, n\}$, instead of using the usual assumption of globally Lipschitz gradients. {A natural extension of our approach is allowing the agents to perform proximal-gradient steps with independent stepsizes, without the need of additional rounds of communication in order to reach consensus on the stepsize in each iteration. Such a method would correspond to an extension of NIDS \cite{li2019decentralized}, a decentralised method for distributed optimisation that naturally allows constant mismatched stepsizes across agents. Devising a decentralised method that allows variable independent stepsizes for the agents is one of the planned future research directions of the authors.}



\noindent {\textbf{Acknowledgements.} We are grateful to the anonymous reviewers for their constructive feedback and valuable suggestions, which significantly contributed to the improvement of this manuscript.}

\noindent\textbf{Funding.} FA, MND and MKT were supported in part by Australian Research Council grant DP230101749.

\noindent\textbf{Conflict of interest.} The authors have no competing interests to declare that are relevant to the content of this article.

{\noindent\textbf{Data availability and replication of results.} The datasets and code used in this work to run the experiments are available at \href{https://github.com/fatenasm/distributed-reconstruction-ls}{github.com/fatenasm/distributed-reconstruction-ls}.}

\bibliographystyle{abbrv}
\bibliography{biblio}

\end{document}